\documentclass[a4paper]{article}

\title{Untwisting the twisted character variety via a stacky curve}
\author{Anna Borri}
\date{}

\usepackage[T1]{fontenc} 

\usepackage[utf8]{inputenc} 

\usepackage[main=english]{babel} 

\usepackage[style=alphabetic]{biblatex}
\usepackage{csquotes}

\usepackage{mathpazo} 

\usepackage{indentfirst} 

\usepackage[scale=0.71]{geometry}

\usepackage{microtype}

\usepackage{graphicx}

\usepackage{tikz-cd}
\usetikzlibrary{arrows}

\usepackage{quiver}

\usepackage{tikz}
\usetikzlibrary{calc, arrows.meta, positioning, patterns}

\usepackage{amsfonts,amssymb}

\usepackage{mathrsfs}

\usepackage{stmaryrd}

\usepackage{amsmath}

\usepackage{mathtools}
\mathtoolsset{showonlyrefs,showmanualtags} 
\numberwithin{equation}{section} 

\usepackage{amsthm}

\usepackage{float}

\usepackage{adjustbox}

\usepackage{enumitem}

\usepackage{verbatim}

\usepackage{comment}

\usepackage{zref-clever} 
\zcsetup{capfirst} 
\usepackage{refcount}

\usepackage{hyperref} 
\hypersetup{
  colorlinks,
  citecolor=magenta,
  linkcolor=black,
  urlcolor=blue}

\AddToHook{env/thm/begin}{%
\zcsetup{countertype={thm=theorem}}}

\AddToHook{env/defi/begin}{%
\zcsetup{countertype={thm=definition}}}

\AddToHook{env/prop/begin}{%
\zcsetup{countertype={thm=proposition}}}

\AddToHook{env/lemma/begin}{%
\zcsetup{countertype={thm=lemma}}}

\AddToHook{env/cor/begin}{%
\zcsetup{countertype={thm=corollary}}}

\AddToHook{env/rmk/begin}{%
\zcsetup{countertype={thm=remark}}}

\newlist{thmlist}{enumerate}{1}
\setlist[thmlist]{label=\upshape(\roman*),ref=\thethm(\roman*)}

\newlist{proplist}{enumerate}{1}
\setlist[proplist]{label=\upshape(\roman*),ref=\theprop(\roman*)}

\zcsetup{
  font=\upshape,
  countertype = {
    thmlisti   = theorem ,
    thmlistii  = theorem ,
    thmlistiii = theorem ,
    thmlistiv  = theorem ,
  } ,
  counterresetby = {
    thmlistii  = thmlisti ,
    thmlistiii = thmlistii ,
    thmlistiv  = thmlistiii ,
  } ,
  countertype = {
    proplisti   = proposition ,
    proplistii  = proposition ,
    proplistiii = proposition ,
    proplistiv  = proposition ,
  } ,
  counterresetby = {
    proplistii  = proplisti ,
    proplistiii = proplistii ,
    proplistiv  = proplistiii ,
  } ,
}
\ExplSyntaxOn
\cs_new_protected:Nn \gaussler_extract_subref:n
  {
    \__gaussler_extract_last:n #1
  }
\cs_new:Nn \__gaussler_extract_last:n
  {
    \__gaussler_extract_last:w #1 () \q_stop
  }
\cs_generate_variant:Nn \gaussler_extract_subref:n {e}
\cs_new:Npn \__gaussler_extract_last:w #1 (#2) #3 \q_stop
 {
  \tl_if_empty:nTF {#2} {\textbf{??}} {(#2)}
 }

\NewDocumentCommand{\localref}{m}
  {
    \hyperref [#1] { \textup { \gaussler_extract_subref:e { \getrefnumber{#1} } } }
  }
\ExplSyntaxOff

\theoremstyle{definition}
\newtheorem{thm}{Theorem}[section]
\newtheorem{defi}[thm]{Definition}

\newtheorem{prop}[thm]{Proposition}
\newtheorem{lemma}[thm]{Lemma}
\newtheorem{cor}[thm]{Corollary}
\newtheorem{rmk}[thm]{Remark}

\newtheorem*{thm*}{Theorem}
\newtheorem*{defi*}{Definition}

\newtheorem{introthm}{Theorem}       
\DeclareMathOperator{\Hom}{Hom}

\DeclareMathOperator{\Aut}{Aut}

\DeclareMathOperator{\qcoh}{QCoh}
\DeclareMathOperator{\Spec}{Spec}

\DeclareMathOperator{\rank}{rk}

\DeclareMathOperator{\Higgs}{Higgs}

\DeclareMathOperator{\id}{id}
\DeclareMathOperator*{\holim}{holim}

\newcommand{\inHom}{\underline{\Hom}}

\newcommand{\Cb}{\mathbb{C}}

\newcommand{\Zb}{\mathbb{Z}}
\newcommand{\Qb}{\mathbb{Q}}

\newcommand{\Ab}{\mathbb{A}}
\newcommand{\Gb}{\mathbb{G}}
\newcommand{\Sb}{\mathbb{S}}
\newcommand{\Db}{\mathbb{D}}

\newcommand{\mc}[1]{\mathcal{#1}}
\renewcommand{\O}{\mc{O}}
\newcommand{\shHom}{\mc{H}\kern -.5pt om}
\newcommand{\shEnd}{\mc{E}\kern -.5pt nd}
\newcommand{\shExt}{\mc{E}\kern -.5pt xt}
\newcommand{\Bun}{\mc{B}\kern -.5pt un}

\newcommand{\Ec}{\mc{E}}

\newcommand{\Vc}{\mc{V}}

\newcommand{\Xc}{\mc{X}}

\newcommand{\mymid}{\;\middle|\;}
\newcommand{\dd}{\mathop{}\!\mathrm{d}} 
\DeclareMathSymbol{\mh}{\mathord}{operators}{`\-} 

\newcommand{\modd}{/\!\!/} 

\newcommand{\XX}{\mc{X}} 

\newcommand{\RH}{R\!H} 

\newcommand{\catname}[1]{\mathbf{#1}}

\newcommand{\Cov}{\catname{Cov}}
\newcommand{\Set}{\catname{Set}}
\newcommand{\SSet}{\catname{SSet}}

  \newcommand{\graphfragment}[1]{
    \begin{tikzpicture}[baseline=(v0.base), scale=0.8]
      \node[circle, fill, inner sep=1pt, label=below:{\tiny $v_0$}] (v0) at (0,0) {};
      \node[circle, fill, inner sep=1pt, label=below:{\tiny $v_1$}] (v1) at (0.8,0) {};
      \node[circle, fill, inner sep=1pt, label=below:{\tiny $v_n$}] (vn) at (1.8,0) {};
      \node[circle, fill, inner sep=1pt, label=below:{\tiny $v_0$}] (v0end) at (2.6,0) {};
      
      \draw[-{Stealth}] (v0) -- node[above] {\tiny $e_0$} (v1);
      \draw[dashed] (v1) -- (vn);
      \draw[-{Stealth}] (vn) -- node[above] {\tiny $e_n$} (v0end);
    \end{tikzpicture}
  }

\tikzcdset{
  every diagram/.append style={
    cells={font=\everymath\expandafter{\the\everymath\displaystyle}}
  }
}

\begin{document}
\maketitle

\begin{abstract}
Given a smooth projective curve $X$ over $\Cb$, we show that the twisted character variety of $X$ is equivalent to a component of the character variety of a stacky curve $\XX$, constructed by adding an orbifold point to $X$. This allows to interpret the non abelian Hodge correspondence in degree $d$ over $X$ as a statement about the geometry of the Hodge moduli space of $\lambda$-connections on $\XX$. Following Shende's approach (see arXiv:1411.4975), we exploit this interpretation to compute the weights of the tautological classes on the cohomology of the twisted character variety, by replacing $\XX$ with a simplicial scheme $\Delta_{\XX}$ obtained from a triangulation.
\end{abstract}

\section*{Introduction}
Simpson's non abelian Hodge correspondence establishes a homeomorphism between the moduli space of semistable Higgs bundles on a smooth projective curve $X$ and the character variety of the curve. This can be understood as relating the geometry of the generic fiber and the special fiber inside the moduli space of $\lambda$-connections.
Namely, for a complex parameter $\lambda$ and a vector bundle $\mc{E}$ on $X$, define a $\lambda$-connection on $\mc{E}$ to be a $\Cb$-linear operator
\begin{equation}
\nabla \colon \mc{E} \to \mc{E}\otimes \Omega_X
\end{equation}
satisfying the Leibniz rule up to the scalar $\lambda$. Explicitly, given sections $f \in \O_X(U)$ and $s \in \mc{E}(U)$ we have
\begin{equation}
\nabla(f \, s) = f \, \nabla(s) + \lambda \, s \otimes \dd f.
\end{equation}
The Hodge moduli space $\mc{M}_{n, Hdg}^0(X)$ parametrizes triples $(\mc{E}, \lambda, \nabla)$, where $\mc{E}$ is a rank $n$ and degree $0$ vector bundle on $X$ and $\nabla$ is a $\lambda$-connection on $\mc{E}$, subject to a stability condition.
The assignment
\begin{equation}
(\mc{E}, \lambda, \nabla) \mapsto \lambda
\end{equation}
gives $\mc{M}_{n, Hdg}^0(X)$ the structure of a family over $\Ab^1$, trivial over $\Gb_m$.
Notice that the generic fiber (identified with $\lambda = 1$) is the moduli space of vector bundles on $X$ with a (flat) connection, i.e. the De Rham moduli space $\mc{M}_{n, DR}^0(X)$. The special fiber is the moduli space of semistable Higgs bundles on $X$, i.e. the Dolbeault moduli space $\mc{M}_{n, Dol}^0(X)$.

The Riemann-Hilbert correspondence gives a complex-analytic isomorphism between the De Rham moduli space and the moduli space of rank $n$ local systems on $X$, which can be identified with the character variety
\begin{equation}
\mc{M}_{n, Betti}^0(X) \coloneqq \left[ \Hom(\pi_1(X), GL_n) / GL_n \right].
\end{equation}

In short, we have the following picture
\begin{equation} \label{eq: pullback diagram}
    \begin{tikzcd}[ampersand replacement=\&, column sep=large, row sep=huge, baseline=(current bounding box.center)]
        {\mc{M}_{n, Dol}^0(X)} \& {\mc{M}_{n,Hdg}^0(X)} \& {\mc{M}_{n, DR}^0(X)} \& {\mc{M}_{n,B}^0(X)} \\
        {\{ 0 \}} \& {\Ab^1} \& {\{ 1 \}} 
        \arrow[hook, from=1-1, to=1-2]
        \arrow[from=1-1, to=2-1]
        \arrow["\lrcorner"{anchor=center, pos=0.125}, draw=none, from=1-1, to=2-2]
        \arrow[from=1-2, to=2-2]
        \arrow[hook', from=1-3, to=1-2]
        \arrow["\cong", from=1-3, to=1-4]
        \arrow["\lrcorner"{anchor=center, pos=0.125, rotate=-90}, draw=none, from=1-3, to=2-2]
        \arrow[from=1-3, to=2-3]
        \arrow[from=2-1, to=2-2]
        \arrow[from=2-3, to=2-2]
    \end{tikzcd}
\end{equation}
and Simpson's correspondence shows that the analytic spaces underlying $\mc{M}_{n, Dol}^0(X)$ and $\mc{M}_{n, DR}^0(X)$ are homeomorphic. From an algebraic perspective, the main consequence is that these varieties have isomorphic cohomology rings. Indeed, we get that the inclusions of the fibers in the total space induce isomorphisms of cohomology rings
\begin{equation}
H^*(\mc{M}_{n, Dol}^0(X)) \cong H^*(\mc{M}_{n,Hdg}^0(X)) \cong H^*(\mc{M}_{n, DR}^0(X));
\end{equation}
and that the analytic isomorphism $\mc{M}_{n, DR}^0(X) \cong \mc{M}_{n,B}^0(X)$ induces an isomorphism
\begin{equation}
H^*(\mc{M}_{n, DR}^0(X)) \cong H^*(\mc{M}_{n,B}^0(X)).
\end{equation}

Notice that so far we have described only moduli spaces of degree $0$ vector bundles. Indeed, for a vector bundle to admit a flat connection it needs to have vanishing Chern classes, so the description given of $\mc{M}_{DR}$ only makes sense in degree $0$. However, the moduli space of Higgs bundles $\mc{M}_{Dol}$ can be constructed in the same way for any degree $d$. Classically, this incompatibility has been solved by replacing the De Rham and Betti moduli spaces with \textit{twisted} versions. Namely
$\mc{M}_{n, DR}^d(X)$
is the moduli space of semistable vector bundles with a logarithmic connection with a single pole $p \in X$ of multiplicity at most $n$ and residue $- \tfrac{2 \pi i d}{n}$, and
\begin{equation} \label{eq: twisted betti}
\mc{M}_{n, B}^d(X) = \left[ \bigg\{ (A_i, B_i)_{i = 1}^g \in GL_n^{2g} \mymid \prod_{i=1}^{g}[A_i, B_i] = \mathrm{e}^{-\frac{2 \pi \mathrm{i}}{n} d} Id \bigg\} / GL_n \right]
\end{equation}
is the \textit{twisted character variety}.

The main purpose of this article is to give an alternative description of $\mc{M}_{n, DR}^d(X)$ and $\mc{M}_{n, B}^d(X)$, showing that they can also be realized as generic and special fibers in a family of $\lambda$-connections, in analogy with \eqref{eq: pullback diagram}.

In order to do this, the main insight is that the meromorphic structure around the point $p \in X$ can be interpreted as to be holomorphic around an orbifold point and similarly the twist appearing in $\mc{M}_{n,B}^d(X)$ can be interpreted as monodromy around an orbifold point. Algebraically, the way to realize this orbifold structure is to construct a stacky curve $\XX$ as the $n$-th root stack of $X$ at the point $p$, i.e. as
\begin{equation}
    \begin{tikzcd}[ampersand replacement=\&, column sep=huge, row sep=huge]
            {\XX= \sqrt[n]{p/X} } \&\& {[\Ab^1/\Gb_m]} \\
            X \& {} \& {[\Ab^1/\Gb_m].}
            \arrow[from=1-1, to=1-3]
            \arrow["\pi"', from=1-1, to=2-1]
            \arrow["\lrcorner"{anchor=center, pos=0.125}, draw=none, from=1-1, to=2-2]
            \arrow["{\wedge n}", from=1-3, to=2-3]
            \arrow["{(\O(p), \; s_p)}"', from=2-1, to=2-3]
    \end{tikzcd}
\end{equation}
Since degree $0$ vector bundles on $\XX$ admit $\lambda$-connections, to construct a version of the Hodge moduli space for degree $d$ vector bundles on $X$ we show that these can be realized as degree $0$ vector bundles on $\XX$. Similarly, we show that the Betti moduli space $\mc{M}_{n, B}^d(X)$ can be realized as moduli space of representations of $\pi_1(\XX)$. In this way, comparing the generic and special fiber in the Hodge moduli space of degree $0$ vector bundles on $\XX$, we get a version of the diagram \eqref{eq: pullback diagram} for degree $d$ vector bundles on $X$.
In what follows we give a more detailed description of how this insight is formalized.

Vector bundles on $\XX$ have a third numerical invariant beyond rank and degree, namely the representation obtained as stalk at the stacky point. We denote by $w(-d)$ the weight $-d$ character of $\mu_n$ and by $n \cdot w(-d)$ the corresponding $n$-dimensional representation. Then the main result comparing moduli spaces on $X$ and on $\XX$ is the following.
\begin{introthm}\label{intro thm: equivalence Higgs bundles}
The twisted pullback and pushforward along the moduli map $\mc{X} \to X$ define equivalences of algebraic stacks
\begin{equation}
\mc{M}_{n, Dol}^d(X) \cong \mc{M}_{n \cdot w(-d), Dol}^0(\XX).
\end{equation}
\end{introthm}

Passing to the study of the Betti moduli space, a first observation is that the replacement of $p$ by an orbifold point changes the fundamental group of $X$ by adding a cyclic monodromy element. Namely, we get
\begin{equation} \label{introeq: fundamental group of stacky curve}
\pi_1(\Xc) = \left\langle a_i, \, b_i, \, c \mymid \prod_{i = 1}^g [a_i, b_i] = c, \; c^n = 1 \right\rangle .
\end{equation}
Then we get the twisted character variety thanks to the following theorem.
\begin{introthm} \label{intro thm: twisted character variety}
The connected component of the Betti moduli stack of $\mc{X}$ corresponding to the numerical invariant $n \cdot w(-d)$ is isomorphic to the twisted Betti moduli space of $X$, i.e. we have
\begin{equation}
\mc{M}_{n \cdot w(-d), B}^0(\mc{X}) = \left[ \bigg\{ (A_i, B_i)_{i = 1}^g \in GL_n^{2g} \mymid \prod_{i=1}^{g}[A_i, B_i] = \mathrm{e}^{-\frac{2 \pi \mathrm{i}}{n} d} Id \bigg\} / GL_n \right].
\end{equation}
\end{introthm}

Using $\lambda$-connections, we can again construct a family over $\Ab^1$ whose special and generic fiber are the Dolbeault and De Rham moduli spaces. Explicitly, we get the following diagram
\begin{equation}
\begin{tikzcd}[ampersand replacement=\&, column sep=large, row sep=huge]
        {\mc{M}_{n \cdot w(-d), Dol}^0(\XX)} \& {\mc{M}_{n \cdot w(-d),Hdg}^0(\XX)} \& {\mc{M}_{n \cdot w(-d), DR}^0(\XX)} \& {\mc{M}_{n \cdot w(-d),B}^0(\XX)} \\
        {\{ 0 \}} \& {\Ab^1} \& {\{ 1 \},}
        \arrow[hook, from=1-1, to=1-2]
        \arrow[from=1-1, to=2-1]
        \arrow["\lrcorner"{anchor=center, pos=0.125}, draw=none, from=1-1, to=2-2]
        \arrow[from=1-2, to=2-2]
        \arrow[hook', from=1-3, to=1-2]
        \arrow["\cong", from=1-3, to=1-4]
        \arrow["\lrcorner"{anchor=center, pos=0.125, rotate=-90}, draw=none, from=1-3, to=2-2]
        \arrow[from=1-3, to=2-3]
        \arrow[from=2-1, to=2-2]
        \arrow[from=2-3, to=2-2]
    \end{tikzcd}
\end{equation}
where the isomorphism
\begin{equation}
\mc{M}_{n \cdot w(-d), DR}^0(\XX) \xlongrightarrow{\cong} \mc{M}_{n \cdot w(-d),B}^0(\XX)
\end{equation}
is complex analytic but not algebraic. As before, we get that the inclusions of the fibers in the total space induce isomorphisms of cohomology rings
\begin{equation}
H^*(\mc{M}_{n \cdot w(-d), Dol}^0(\XX)) \cong H^*(\mc{M}_{n \cdot w(-d),Hdg}^0(\XX)) \cong H^*(\mc{M}_{n \cdot w(-d), DR}^0(\XX));
\end{equation}
and that the analytic isomorphism $\mc{M}_{n \cdot w(-d), DR}^0(\XX) \cong \mc{M}_{n \cdot w(-d),B}^0(\XX)$ induces an isomorphism
\begin{equation}
H^*(\mc{M}_{n \cdot w(-d), DR}^0(\XX)) \cong H^*(\mc{M}_{n \cdot w(-d),B}^0(\XX)).
\end{equation}

This approach has the further advantage of simplifying the calculation of the Hodge weights of the tautological classes on $\mc{M}_{n, B}^d(X)$. In \cite[]{Shende-weights}, Shende computed the weights of the tautological classes on the character variety of a topological space by replacing the space by a triangulation and then realizing this triangulation as a simplicial scheme. He then computed the weights also in the case of twisted character varieties, using a reduction to the case of $GL_n$- and $PGL_n$-character varieties presented in \cite[\S 2.2]{MixedHodgePolynomials}. The following is the result of this computation.
\begin{thm}[Shende] \label{thm shende}
Under the isomorphism $ H^*(\mc{M}_{n, DR}^0(X)) \cong H^*(\mc{M}_{n, B}^0(X))$, the Künneth components $e_{i,j}$ of the $i$-th Chern class of the universal bundle $c_i(\mc{E}_{\text{univ}})$ are mapped to cohomology classes of weight $2i$. The same holds true in the twisted case, under the isomorphism $ H^*(\mc{M}_{n, DR}^d(X)) \cong H^*(\mc{M}_{n, B}^d(X))$.
\end{thm}

Our approach gives an alternative direct proof for the twisted case, which doesn't use the case of $PGL_n$-bundles. Indeed, using Theorem~\ref{intro thm: twisted character variety} we see that the twisted character variety can be realized as an actual variety of local systems on $\XX$. Therefore Shende's argument can be extended directly to cover this case. The main ingredient is the construction of a simplicial scheme $\Delta_{\XX}$ which realizes a triangulation of the stacky curve $\XX$. In the article we explain in detail the construction of this simplicial decomposition. Once this is done, the argument amounts to using the equivalence
\begin{equation}
\mc{M}_{n, B}^0(\XX)=\left[ \text{Hom}(\pi_1(\XX), GL_n) / GL_n \right] \cong \shHom_{SSch}(\Delta_{\XX}, BGL_n)
\end{equation}
and the universal object on $\shHom_{SSch}(\Delta_{\XX}, BGL_n)$ to carry out the computation of the weights. Using this argument we get the following theorem.
\begin{introthm}
The cohomology classes on $\mc{M}_{n, B}^0(\XX)$ obtained by pulling back the $i$-th tautological classes on $\mc{M}_{n, DR}^0(\XX)$ have pure weight $2i$.
\end{introthm}

By restricting to the $n \cdot w(-d)$-component of the Betti moduli space we get an alternative proof of \zcref{thm shende}.

We now briefly explain the structure on the paper. 
In \zcref{section: vector bundles} we construct the stacky curve $\XX$ and recall the numerical invariants decomposing the moduli space of vector bundles on $\XX$ into connected components. Finally, we show a version of Theorem \ref{intro thm: equivalence Higgs bundles} for moduli of vector bundles.\\
In \zcref{section: Higgs bundles and Dolbeault moduli space} we pass from the study of vector bundles to the study of Higgs bundles, proving Theorem \ref{intro thm: equivalence Higgs bundles}.\\
In \zcref{section: De Rham and Betti moduli spaces} we study the De Rham and Betti side of the non abelian Hodge correspondence. First we show that the fundamental group $\pi_1(\XX)$ has the expected presentation \eqref{introeq: fundamental group of stacky curve}.
Then we use the monodromy action on a flat bundle to understand how the numerical invariants on $\mc{M}_{DR}(\XX)$ translate on the Betti side, proving Theorem \ref{intro thm: twisted character variety}.\\
In \zcref{section: The non abelian Hodge correspondence} we introduce the Hodge moduli space on $\XX$ and realize $\mc{M}_{Dol}(\XX)$ and $\mc{M}_{DR}(\XX)$ as special and generic fibers of a family over $\Ab^1$. We then use this description to compare the cohomology of $\mc{M}_{Dol}(\XX)$, $\mc{M}_{DR}(\XX)$ and $\mc{M}_{B}(\XX)$.\\
Finally, in \zcref{section: The Hodge weights} we explain how the cohomological isomorphisms described in \zcref{section: The non abelian Hodge correspondence} interact with the Hodge weights. We first recall Shende's argument, explaining in detail how the triangulation of a complex manifold is used to compute the weights of the tautological classes on the Betti moduli space. Then we generalize this argument to our case, by constructing a triangulation of the stacky curve $\XX$.
\section{The moduli space of vector bundles} \label{section: vector bundles}
We start by proving a version of Theorem \ref{intro thm: equivalence Higgs bundles} for the moduli stacks of vector bundles on a smooth projective curve.
Let $X$ be a smooth, projective, genus $g$ curve over $\Cb$ and let $p \in X$ be a closed point. Following \cite[]{Cadman}, we construct $\mc{X} = \sqrt[n]{p/X} $ as a root stack over $X$, with multiplicity $n$ in $p$. We then show that vector bundles of rank $n$ and degree $d$ on $X$ can be identified with vector bundles of rank $n$ and degree $0$ on $\XX$ with a fixed constant structure at the stacky point.

Consider $p$ as an effective Cartier divisor on $X$. Associated to it, there is a pair $(\O_X(p), s_p)$, where $\O_X(p)$ is a line bundle and $s_p$ is a section of $\O_X(p)$ vanishing at $p$, called the \textit{tautological section}. This pair defines a morphism
\[
X \xrightarrow{(\O(p), \; s_p)} [\Ab^1/\Gb_m].
\]
Let $\wedge n \colon [\Ab^1/\Gb_m] \to [\Ab^1/\Gb_m]$ be the morphism induced by taking the $n$-th power maps on $\Ab^1$ and on $\Gb_m$. Define $\mc{X}$ as the pullback
\begin{equation} \label{eq: definition of X}
\begin{tikzcd}
	\Xc && {[\Ab^1/\Gb_m]} \\
	X & {} & {[\Ab^1/\Gb_m].}
	\arrow[from=1-1, to=1-3]
	\arrow["\pi"', from=1-1, to=2-1]
	\arrow["\lrcorner"{anchor=center, pos=0.125}, draw=none, from=1-1, to=2-2]
	\arrow["{\wedge n}", from=1-3, to=2-3]
	\arrow["{(\O(p), \; s_p)}"', from=2-1, to=2-3]
\end{tikzcd}
\end{equation}
The vertical map $\pi$ is a coarse moduli map, see \cite[Corollary 2.3.7]{Cadman}.
The stack $\Xc$ comes equipped with a pair $(\O(\frac{1}{n}p), \frac{1}{n}s_p)$, where $\O(\frac{1}{n}p)$ is a line bundle and $\frac{1}{n}s_p$ is a section of $\O(\tfrac{1}{n}p)$. This pair corresponds to the upper horizontal morphism in \eqref{eq: definition of X}, and is characterized by the properties
\begin{equation} \label{eq: pullback of O(p)}
\pi^*\O(p) = \O(\tfrac{1}{n}p)^{\otimes n}, \quad \pi^* s_p = (\tfrac{1}{n}s_p)^{\otimes n}.
\end{equation}
The sheaf $\O(\frac{1}{n}p)$ is called the \textit{tautological sheaf} of the stacky curve $\Xc$. It corresponds to the composition
\[
\Xc \to [\Ab^1 / \Gb_m] \to B\Gb_m.
\]

With an abuse of notation, we denote by $p$ also the stacky point in $\mc{X}$ above $p$. Notice that in $p$ the stack $\mc{X}$ has the structure
\begin{equation}\label{punctual structure X}
\begin{tikzcd}
	{B\mu_n} & \Xc \\
	{\{p\}} & X.
	\arrow[hook, from=1-1, to=1-2]
	\arrow[from=1-1, to=2-1]
	\arrow["\lrcorner"{anchor=center, pos=0.125}, draw=none, from=1-1, to=2-2]
	\arrow["\pi", from=1-2, to=2-2]
	\arrow[hook, from=2-1, to=2-2]
\end{tikzcd}
\end{equation}
Away from $p$ the morphism $\pi$ is an isomorphism
\begin{equation}\label{X-p}
\mc{X} \setminus \{ p \} \cong X \setminus \{p\}.
\end{equation}

We now want to understand the numerical invariants that decompose the moduli space of vector bundles on $\XX$ into connected components. Let $\mc{E}$ be a rank $m$ vector bundle on $\XX$. By pulling back along the inclusion of the stacky point
\begin{equation}
i \colon B \mu_n \hookrightarrow \XX
\end{equation}
we get a vector bundle on $B \mu_n$, i.e. an $m$-dimensional representation of the cyclic group $\mu_n$, which decomposes as a direct sum of characters.
Denoting by $w(i)$ the weight $i$ character, we get a decomposition
\[
i^* \mc{E} = \bigoplus_{i \in \Zb / n \Zb} w(i)^{m_{i}}.
\]
\begin{defi}\cite[Definition 1.2.7]{Taams}
To keep track of the multiplicities of all of the weights appearing, we denote the multiplicity vector of $\mc{E}$ at $p$ by
\[
\underline{m}(\mc{E}) \coloneqq m_0 \, w(0) + \dots + m_{n-1} \, w(n-1).
\]
\end{defi}
Notice that the rank is recovered from the multiplicity vector as $m = m_0 + \dots + m_{n-1}$.
By construction the tautological line bundle $\O(\tfrac{1}{n}p)$ has weight $1$ in $p$, i.e. we have
\begin{equation}
\underline{m}(\O(\tfrac{1}{n}p)) = 1 \cdot w(1).
\end{equation}

\begin{rmk}\cite[\S 1.2]{Taams} \label{tensor shift}
Since pullback commutes with tensor product and $w(i) \otimes w(j) = w(i+j)$, we see that tensoring with the tautological line bundle acts as a shift operator on the multiplicity vector, mapping $w(i)$ to $w(i+1)$.
\end{rmk}

The other numerical invariant associated to a vector bundle on a curve is its degree. We need to define a notion of degree for vector bundles on $\XX$, respecting the classical formulas for the degree of the direct sum and tensor product of two vector bundles. Keep in mind that we want to identify vector bundles on $X$ with their pullback along $\pi$, so we want
\begin{equation}
\pi^* \colon \Bun_n(X) \to \Bun_n(\XX)
\end{equation}
to be degree preserving. It turns out that this property completely determines the definition of the degree.

\begin{defi}(See \cite[Definition 1.2.29]{Taams}).
Let $K(\XX)$ denote the Grothendieck group of the stacky curve $\XX$. Then the degree is defined to be the unique group homomorphism
\begin{equation}
\deg\colon K(\XX) \to \Qb
\end{equation}
such that the pullback $\pi^*\colon K(X) \to K(\XX)$ is degree preserving.
\end{defi}
See \cite[\S 1.2]{Taams} for a more explicit definition. From \zcref{eq: pullback of O(p)} we get
\begin{equation}
\deg(\O(\tfrac{1}{n}p)) = \tfrac{1}{n},
\end{equation}
which further justifies the notation $\O(\tfrac{1}{n}p)$.

Since we want to identify vector bundles on $X$ with their pullback along $\pi$, we mention the main properties of the pullback and pushforward functors $\pi^*$ and $\pi_*$.
\begin{prop} \label{properties of pullback and pushforward}
The pullback and pushforward functors along $\pi$ have the following properties.
\begin{proplist}
    \item The pushforward $\pi_* \colon \qcoh(\XX) \to \qcoh(X)$ is exact and restricts to a functor $\pi_* \colon \Bun(\XX) \to \Bun(X)$. \cite[Theorem 2.3.4]{abramovich-vistoli}
    \item The pullback $\pi^* \colon \qcoh(X) \to \qcoh(\XX)$ is exact. \cite[Proposition 1.2.3]{Taams}.
	\item If $\mc{E}$ is a vector bundle on $\XX$ with multiplicity $\underline{m}(\mc{E}) = m_0 \, w(0) + \dots + m_{n-1} \, w(n-1)$, then
	\begin{equation}
	\deg \mc{E} = \deg(\pi_* \mc{E}) + \sum_{i=0}^{n-1} i \,  m_i ,
	\end{equation}
	see \cite[Proposition 1.2.30]{Taams}. \zlabel{formula degree-multiplicities}
	\item If $F$ is a coherent sheaf on $X$, then $\underline{m}(\pi^* F) = n \cdot w(0)$. Namely, the pullback $\pi^* F$ is endowed with the trivial $\mu_{n}$ action at $p$. \cite[Example 1.2.9]{Taams} \zlabel{multiplicity of pullback}
\end{proplist}
\end{prop}

\begin{thm}\cite[Corollary 3.1.9, Theorem 3.1.12]{Taams} \label{connected components of the moduli stack}
The stack $\Bun(\XX)$ is a smooth algebraic stack. Its connected components are the substacks
\begin{equation}
\Bun_{\underline{m}}^{d}(\XX) \coloneqq \left< \mc{E} \in \Bun(\XX) \mymid ( \underline{m}(\mc{E}), \deg \mc{E}) = (\underline{m}, d)\right>,
\end{equation}
parametrizing vector bundles of fixed degree $d \in \tfrac{1}{n}\Zb$ and multiplicity vector $\underline{m} = m_0 \, w(0) + \dots + m_{n-1} \, w(n-1)$.
\end{thm}

We are finally ready to compare the moduli spaces of vector bundles on $X$ with moduli spaces of vector bundles on $\XX$.

\begin{prop} \label{pullback and pushforward are equivalences}
The pullback and pushforward functors along the moduli map $\pi\colon \XX \to X$ define inverse equivalences of algebraic stacks
\begin{equation}\label{eq: eqv Bun}
\begin{tikzcd}
	{\Bun_{n}^d(X)} & {\Bun_{n \cdot w(0)}^d(\XX),}
	\arrow[""{name=0, anchor=center, inner sep=0}, "{\pi^*}", curve={height=-18pt}, from=1-1, to=1-2]
	\arrow[""{name=1, anchor=center, inner sep=0}, "{\pi_*}", curve={height=-18pt}, from=1-2, to=1-1]
	\arrow["\cong"{description}, draw=none, from=0, to=1]
\end{tikzcd}
\end{equation}
which preserve (semi)stability.
\end{prop}
\begin{proof}
The pullback $\pi^*$ is degree-preserving by definition and by \zcref{multiplicity of pullback} the multiplicity of $\pi^*E$ is $n \cdot w(0)$ for all $E \in \Bun_n(\XX)$. Dually, if the multiplicity of $\mc{E} \in \Bun^d(\XX)$ is $n \cdot w(0)$, then \zcref{properties of pullback and pushforward} shows that $\pi_*\mc{E}$ is a vector bundle of degree $d$. This shows that the functors in \eqref{eq: eqv Bun} are well defined. The two compositions are naturally isomorphic to the identity by \cite[Corollary 1.2.6, Theorem 1.2.10]{Taams}.\\
Finally, since rank and degree are preserved, so is (semi)stability.
\end{proof}

On $\XX$ we can twist with the tautological line bundle to change the degree of a vector bundle. In this way we achieve the goal of realizing degree $d$ vector bundles on $X$ as degree $0$ vector bundles on $\XX$.

\begin{thm} \label{thm: equivalence for Bun}
The \textit{twisted} pullback and pushforward functors, defined by
\begin{equation} \label{equivalence bun}
\begin{tikzcd}
	{\Bun_n^d(X)} &&& {\Bun_{n \cdot w(-d)}^0(\mc{X})}
	\arrow["{\pi^*(-) \otimes \O(-\tfrac{d}{n}p)}", shift left=2, from=1-1, to=1-4]
	\arrow["{\pi_*(\O(-\tfrac{d}{n}p) \otimes -)}", shift left=2, from=1-4, to=1-1]
\end{tikzcd}
\end{equation}
define inverse equivalences of algebraic stacks. These equivalences preserve (semi)stability, hence they restrict to equivalences
\[\begin{tikzcd}
	{\Bun_n^{d, \, ss}(X)} &&& {\Bun_{n \cdot w(-d)}^{0, \, ss}(\mc{X}).}
	\arrow["{\pi^*(-) \otimes \O(-\tfrac{d}{n}p)}", shift left=2, from=1-1, to=1-4]
	\arrow["{\pi_*(\O(-\tfrac{d}{n}p) \otimes -)}", shift left=2, from=1-4, to=1-1]
\end{tikzcd}\]
\end{thm}
\begin{proof}
Using \zcref{pullback and pushforward are equivalences}, we only need to check that tensoring with $\O(-\tfrac{d}{n}p)= \O(\tfrac{1}{n}p)^{\otimes-d}$ gives an isomorphism
\begin{equation}
\Bun_{n \cdot w(0)}^d(\XX) \xrightarrow{\cong} \Bun_{n \cdot w(-d)}^0(\XX),
\end{equation}
namely that $- \otimes \O(-\tfrac{d}{n}p)$ acts on the numerical invariants as prescribed. For the multiplicity vector, recall that tensoring with the tautological line bundle $\O(\tfrac{1}{n}p)$ acts as a shift, as observed in \zcref{tensor shift}. For the degree, use the formula for the degree of a product of vector bundles to get
\[
\deg\left( \mc{E} \otimes \O(-\tfrac{d}{n}p) \right) = \deg \mc{E} + n \left( -\frac{d}{n} \right) = 0.
\]
Finally, notice that
\[
\mu(\mc{E} \otimes \O(-\tfrac{d}{n}p)) = \frac{\deg \mc{E} - \rank \mc{E} \frac{d}{n}}{\rank \mc{E}} = \mu(\mc{E}) - \frac{d}{n},
\]
showing that tensoring with $\O(-\tfrac{1}{n}p)$ acts as a shift on the slope. Therefore (semi)stability is also preserved.
\end{proof}
\section{Higgs bundles and Dolbeault moduli space}\label{section: Higgs bundles and Dolbeault moduli space}
In order to relate the non abelian Hodge correspondence on $X$ to the correspondence on $\XX$, we need to extend \zcref{thm: equivalence for Bun} to an equivalence of stacks of Higgs bundles. Then after passing to the semistable locus we will get Theorem \ref{intro thm: equivalence Higgs bundles}, namely an equivalence
\begin{equation}
\mc{M}_{n, Dol}^d(X) \cong \mc{M}_{n \cdot w(-d), Dol}^0(\XX).
\end{equation}
We explain how the functors $\pi^*$ and $\pi_*$ extend to the Higgs fields.
A Higgs bundle on $X$ is a pair $(E, \varphi)$ where $E$ is a vector bundle on $X$ and $\varphi$ is an $\O_X$-linear operator
\begin{equation}
\varphi\colon E \to E \otimes \Omega_X,
\end{equation}
called the \textit{Higgs field}.
After pulling back, we get an operator
\begin{equation}
\pi^* \varphi\colon \pi^* E \to \pi^* E \otimes \pi^* \Omega_X.
\end{equation}
To get a Higgs field on $\pi^* E$ we compose with the canonical morphism $\pi^* \Omega_X \to \Omega_{\XX}$. 
This defines a functor
\begin{align}
	(\pi^*, \pi^{\#}) \colon \Higgs_n^d(X) &\to \Higgs_{n \cdot w(0)}^d(\mc{X}) \\
	(E, \varphi) &\mapsto (\pi^*E, \pi^{\#}\varphi).
\end{align}
In order to define an inverse functor, we give an explicit expression for the canonical morphism $\pi^* \Omega_X \to \Omega_{\XX}$.
\begin{lemma}\label{omega mc(X)}
The canonical morphism $\pi^* \Omega_X \to \Omega_{\XX}$ extends along the inclusion
\begin{equation}
n \id_{\pi^* \Omega_X} \otimes s^{n-1} \colon \pi^* \Omega_X \hookrightarrow \pi^* \Omega_X \otimes \O(\tfrac{n-1}{n}p)
\end{equation}
to an isomorphism $\pi^* \Omega_X \otimes \O(\tfrac{n-1}{n}p) \cong \Omega_{\XX}$.
\end{lemma}
\begin{proof}
Away from the point $p$ the statement is trivial. Let $U \subseteq X$ be a neighborhood of $p$ admitting an étale morphism $U \to \Ab^1$ and let $\mc{U} \coloneqq \pi^{-1}(U)$ be the corresponding open in $\XX$. We have a cartesian square
\begin{equation}
\begin{tikzcd}
	{\mc{U}} && {[\Ab^1 / \mu_n]} \\
	U & {\Ab^1} & {\Ab^1 / \mu_n,}
	\arrow[from=1-1, to=1-3]
	\arrow["\pi"', from=1-1, to=2-1]
	\arrow["\lrcorner"{anchor=center, pos=0.125}, draw=none, from=1-1, to=2-2]
	\arrow["\pi'", from=1-3, to=2-3]
	\arrow[from=2-1, to=2-2]
	\arrow["\cong"{description}, draw=none, from=2-2, to=2-3]
\end{tikzcd}
\end{equation}
where the horizontal arrows are étale and hence induce isomorphisms on the cotangent sheaves via pullback. Therefore, it suffices to prove the statement after replacing $\pi$ by $\pi'$.
Composing with the atlas $\Ab^1 \to [\Ab^1 / \mu_n]$, we get the following diagram
\[\begin{tikzcd}
	{\Ab^1 = \Spec k[x]} & \\
	{[\Ab^1 / \mu_n]} \\
	{\Ab^1 / \mu_n} & {\Spec k[x^n] = \Spec k[t].}
	\arrow["{q'}", from=1-1, to=2-1]
	\arrow["{x \mapsto t = x^n}"{pos=0.6}, "q = \pi' \circ q'"' sloped, curve={height=-24pt}, from=1-1, to=3-2]
	\arrow["\pi'", from=2-1, to=3-1]
	\arrow["{=}"{description}, draw=none, from=3-1, to=3-2]
\end{tikzcd}\]
Notice that we have the relation
\[
\dd t = \dd x^n = n \, x^{n-1} \, \dd x,
\]
therefore the canonical morphism $q^* \Omega_X \to \Omega_{\Ab^1}$ corresponds to
\begin{align}
	k[x] \otimes_{k[t]} k[t]\dd t \; &\to \; k[x]\dd x \\
	f(x) \otimes \dd t \quad &\mapsto \; f(x) \, n \, x^{n-1} \, \dd x.
\end{align}
Since $q'$ is étale, it induces isomorphisms on the cotangent sheaves. Moreover, the section $x$ corresponds via $q'$ to the tautological section $s$ of $\O(\tfrac{1}{n}p)$. In conclusion, we get that the canonical morphism $\pi^* \Omega_X \to \Omega_{\mc{X}} \cong \pi^*\Omega_X \otimes \O(\tfrac{n-1}{n}p)$ has the indicated form.
\end{proof}

Under this isomorphism $\Omega_{\XX} \cong \pi^*\Omega_X \otimes \O(\tfrac{n-1}{n}p)$, the pulled-back Higgs field $\pi^{\#}\varphi$ can be rewritten as
\begin{multline}
\pi^{\#} \varphi \colon \pi^*E \xrightarrow{\pi^* \varphi} \pi^*(E \otimes \Omega_X) = \pi^*E \otimes \pi^* \Omega_X \\
\xrightarrow{\id \otimes n \hspace{0.3mm} \id \otimes s^{n-1}} \pi^*E \otimes \pi^*\Omega_X \otimes \O(\tfrac{n-1}{n}p) \cong \pi^*E \otimes \Omega_{\mc{X}}.
\end{multline}

To define the inverse functor, consider $(\mc{E}, \eta) \in \Higgs_{n \cdot w(0)}^d(\XX)$. We want to use $\eta$ to define a Higgs field on $\pi_* \mc{E}$.
Using \zcref{pullback and pushforward are equivalences}, we get that $\mc{E} = \pi^*E$, where $E \cong \pi_* \mc{E}$. By definition, $\eta$ is a morphism in $\Hom_{\mc{X}}(\mc{E}, \mc{E}\otimes\Omega_{\mc{X}})$. Now, using \zcref{omega mc(X)}, we find a chain of natural isomorphisms
\begin{align}
\Hom_{\mc{X}}(\mc{E}, \, \mc{E}\otimes\Omega_{\mc{X}}) & \cong \Hom_{\mc{X}}\left( \pi^*E, \, \pi^*E \otimes \pi^*\Omega_X \otimes \O(\tfrac{n-1}{n}p) \right)\\
& \cong \Hom_X\left( E, \, \pi_*(\pi^*(E \otimes \Omega_{\mc{X}}) \otimes \O(\tfrac{n-1}{n}p)) \right) \label{eq: higgs field, adjunction} \\
& \cong \Hom_X\left( E, \, E \otimes \Omega_X \otimes \pi_*\O(\tfrac{n-1}{n}p) \right) \label{eq: higgs field, projection formula} \\
& \cong \Hom_X\left( E, \, E \otimes \Omega_X \right), \label{eq: higgs field, pushforward}
\end{align}
where \eqref{eq: higgs field, adjunction} follows by the adjunction $\pi^* \dashv \pi_*$, \eqref{eq: higgs field, projection formula} by the projection formula and \eqref{eq: higgs field, pushforward} by the isomorphism $\pi_* \O(\tfrac{n-1}{n}p) \cong \O_X$ obtained from \cite[Corollary 1.2.13]{Taams}. \\
Composing the isomorphisms, we get
\begin{align}
	\Hom_X(E, E \otimes \Omega_X) &\xrightarrow{\cong} \Hom_{\mc{X}}(\mc{E}, \mc{E}\otimes \Omega_{\mc{X}}) \\
	\psi \; &\mapsto \; n \, \pi^* \psi \otimes s^{n-1}.
\end{align}
This shows that there exists a unique $\psi \in \Hom_X(E, E \otimes \Omega_X)$ such that $\eta = n \, \pi^* \psi \otimes s^{n-1}$, where in particular we have $\psi = \frac{1}{n} \, \pi_* \eta$. Therefore, we can define
\[
\pi_{\#}\eta \coloneqq \frac{1}{n} \, \pi_* \eta.
\]
This gives a well defined functor $(\pi_*, \pi_{\#})$.
\begin{equation}
(\pi_*, \pi_{\#}) \colon \Higgs_{n \cdot w(0)}^d(\XX) \to \Higgs_n^d(X)
\end{equation}

\begin{prop}\label{lemma Dol}
The functors $(\pi^*, \pi^{\#})$ and $(\pi_*, \pi_{\#})$ define inverse equivalences of categories
\begin{equation} \label{eq: eqv higgs bundles}
\begin{tikzcd}
	{\Higgs_n^d(X)} & {\Higgs^d_{n \cdot w(0)}(\mc{X}),}
	\arrow[""{name=0, anchor=center, inner sep=0}, "{{(\pi^*, \, \pi^{\#})}}", curve={height=-18pt}, from=1-1, to=1-2]
	\arrow[""{name=1, anchor=center, inner sep=0}, "{{(\pi_*, \, \pi_{\#})}}", curve={height=-18pt}, from=1-2, to=1-1]
	\arrow["\cong"{description}, shift right=1.5, draw=none, from=0, to=1]
\end{tikzcd}
\end{equation}
which preserve (semi)stability.
\end{prop}
\begin{proof}
At the level of vector bundles this is \zcref{pullback and pushforward are equivalences}. We need to show that the two compositions act trivially also on the Higgs fields. Indeed we get
\begin{gather}
\varphi \xmapsto{\pi^{\#}} n \, \pi^* \varphi \otimes s^{n-1} \xmapsto{\pi_{\#}} \frac{1}{n} \, \pi_*(n \, \pi^* \varphi \otimes s^{n-1}) = \varphi\\
\eta \xmapsto{\pi_{\#}} \frac{1}{n} \, \pi_* \eta \xmapsto{\pi^{\#}} n \, \pi^*(\frac{1}{n} \, \pi_* \eta) \otimes s^{n-1} = n \, \pi^* \psi \otimes s^{n-1} = \eta,
\end{gather}
which proves the first statement. As in \zcref{pullback and pushforward are equivalences}, (semi)stability is preserved trivially.
\end{proof}

The Dolbeault moduli space is defined to be the moduli stack parametrising semistable Higgs bundles, i.e.
\begin{equation}
\mc{M}_{n, Dol}^d(X) \coloneqq \Higgs_n^{d, \, ss}(X), \quad \mc{M}_{\underline{m}, Dol}^d(\XX) \coloneqq \Higgs_{\underline{m}}^{d, \, ss}(\XX).
\end{equation}
After twisting \eqref{eq: eqv higgs bundles} with the tautological line bundle on $\XX$, we get a version of \zcref{thm: equivalence for Bun} for the Dolbeault moduli space.
\begin{thm} \label{thm: equivalence for Higgs}
The \textit{twisted} pullback and pushforward functors, defined by
\begin{equation} \label{equivalence dol}
\begin{tikzcd}
	{\Higgs_n^d(X)} &&&& {\Higgs^0_{n \cdot w(-d)}(\mc{X}).}
	\arrow["{(\pi^*, \pi^{\#})(-) \otimes \O(-\tfrac{d}{n}p)}", shift left=2, from=1-1, to=1-5]
	\arrow["{(\pi_*, \pi_{\#})(\O(-\tfrac{d}{n}p) \otimes -)}", shift left=2, from=1-5, to=1-1]
\end{tikzcd}
\end{equation}
define inverse equivalences of algebraic stacks. These equivalences preserve (semi)stability, hence they restrict to equivalences
\begin{equation}\label{eq: equivalence of dolbeault moduli spaces}
\begin{tikzcd}
	{\mc{M}_{n, Dol}^d(X)} &&&& {\mc{M}^0_{n \cdot w(-d), Dol}(\mc{X}).}
	\arrow["{(\pi^*, \pi^{\#})(-) \otimes \O(-\tfrac{d}{n}p)}", shift left=2, from=1-1, to=1-5]
	\arrow["{(\pi_*, \pi_{\#})(\O(-\tfrac{d}{n}p) \otimes -)}", shift left=2, from=1-5, to=1-1]
\end{tikzcd}
\end{equation}
\end{thm}
\begin{proof}
As discussed in \zcref{thm: equivalence for Bun}, twisting with $\O(-\tfrac{d}{n}p)$ shifts the weight at $p$ by $-d$. It also shifts the degree by $-d$ and thus the slope by $-\tfrac{d}{n}$, preserving semistability.
\end{proof}

This concluded the proof of Theorem \ref{intro thm: equivalence Higgs bundles}, presented in the introduction. In the next section we shall proceed with the study of the Betti moduli space and the proof of Theorem \ref{intro thm: twisted character variety}.
\section{De Rham and Betti moduli spaces}\label{section: De Rham and Betti moduli spaces}
In the previous section we considered the Dolbeault moduli space of the stacky curve $\XX$. We now come to the study of the two other moduli spaces appearing in the non abelian Hodge correspondence, namely the De Rham and Betti moduli spaces.
These are defined as follows
\begin{gather}
\mc{M}_{n, DR}(\mc{X}) \coloneqq \left< (\mc{E}, \nabla) \mymid \mc{E}\in \Bun_n(\mc{X}), \; \nabla\colon \mc{E} \to \mc{E}\otimes \Omega_{\XX} \text{ is a flat connection} \right> \\
\mc{M}_{n, B}(\mc{X}) \coloneqq [\Hom(\pi_1(\XX), GL_n) / GL_n].
\end{gather}
The Riemann-Hilbert correspondence establishes a natural isomorphism in the analytic topology
\begin{equation}
\mc{M}_{n, DR}(\mc{X})^{\text{an}} \cong \mc{M}_{n, B}(\mc{X})^{\text{an}}.
\end{equation}

As already observed for the moduli space of vector bundles and for the Dolbeault moduli space, also the de Rham moduli space admits a decomposition in connected components
\begin{equation}
\mc{M}_{n, DR}(\mc{X}) = \coprod_{(\underline{m}, d)} \mc{M}_{\underline{m}, DR}^d(\mc{X}),
\end{equation}
with the latter union ranging over the pairs $(\underline{m} = m_0 w(0)+ \dots m_{n-1} w(n-1), \; d)$ such that $\sum_{i=0}^{n-1} m_{i} = n$.
{Recall that for a vector bundle to admit a flat connection it is a necessary condition that $\deg \mc{E} = 0$.} Therefore we get
\begin{equation} \label{eq: decomposition of MDR}
\mc{M}_{n, DR}(\mc{X}) = \coprod_{\underline{m}} \mc{M}_{\underline{m}, DR}^0(\mc{X}).
\end{equation}
The connected component that we will need to consider is $\mc{M}_{n \cdot w(-d), \, DR}^0(\XX)$. In analogy with \zcref{thm: equivalence for Higgs}, we denote this also as $\mc{M}_{n, DR}^d(X)$ and think of this as the De Rham moduli space for degree $d$ vector bundles on $X$.
We will now give a more explicit description of the Betti moduli stack, allowing us to identify the $n \cdot w(-d)$-stratum in $\mc{M}_{n, DR}(\XX)$ with representations of $\pi_1(\XX)$ with a prescribed structure, thus proving Theorem \ref{intro thm: twisted character variety}.

\subsection{The fundamental group}
In order to decompose the Betti moduli space in connected components, we first give an explicit presentation of the fundamental group $\pi_1(\XX)$. As expected, this is obtained from the fundamental group of a curve by adding a cyclic element corresponding to monodromy around the stacky point. For completeness, we report how to find this presentation using a Seifert--Van Kampen decomposition.

Fix a point $x \in \XX \setminus \{p\}$. By $\pi_1(\XX, x)$ we mean the fundamental group of the complex analytic orbifold underlying the stacky curve $\XX$. This is equivalent to considering $\XX$ as a topological stack with respect to the analytic topology and taking Noohi's fundamental group, defined in \cite{Noohi-foundations}.

Since $\XX$ is connected, locally path connected and semilocally $1$-connected, following Noohi's approach we get that $\pi_1(\XX, x)$ can be computed from the covering spaces of $\XX$, as the automorphism group of the fiber functor over $x$, see \cite[{}18.13, 18.19, 18.22]{Noohi-foundations}. This is analogous to Grothendieck's theory of Galois categories, presented in \cite[]{SGA1-IX}.

The advantage of computing the fundamental group using covering spaces is that we easily get a version of the Seifert--Van Kampen theorem. This is completely analogous to the classical setting of Galois categories, as exposed for example in \cite[Theorem 4.7.7]{Douady2020}. However, since an explicit version of this theorem in our setting seems to be missing, in \zcref{appendix: Seifert--Van Kampen Theorem} we report the proof for completeness.

Using this theorem we can compute the fundamental group of the stacky curve $\XX$.
\begin{prop}\label{prop: fundamental group of stacky curve}
The fundamental group of the stacky curve $\mc{X}$ is
\begin{equation} \label{eq: fundamental group of stacky curve}
\pi_1(\XX) = \left\langle a_i, \; b_i, \; c \mymid \prod_{i = 1}^g [a_i, \, b_i] = c, \; c^n = 1 \right\rangle .
\end{equation}
\end{prop}
\begin{proof}
Locally in the analytic topology, the point $p \in X$ admits a neighborhood biholomorphic to the open unit disk $U \cong \Db \subseteq \Cb$.
Then
\begin{equation}
\mc{X}\vert_U \cong [\Db/\mu_n].
\end{equation}
Let then $\mc{U}_1 \coloneqq \mc{X}\vert_U = \pi^{-1}(U)$, $\mc{U}_2 \coloneqq  \mc{X} \setminus \{ p\}$ and take a point $x \in \mc{U}_1 \cap \mc{U}_2 \eqqcolon \mc{V}$.
By the isomorphism $\XX \setminus \left\{ p \right\} \cong X \setminus  \left\{ p \right\}$ observed in \eqref{X-p}, the substack $\mc{U}_2$ is a punctured curve and therefore
\[
\pi_1(\mc{U}_2, x) = \langle a_i, b_i \mid i =1, \dots g \rangle.
\]
In order to compute $\pi_1(\mc{U}_1, x)$, notice that the atlas $q \colon \Db \to [\Db/\mu_n]$. is a covering map. Since $\Db$ is simply connected, this is the universal cover of $\mc{U}_1$. By \cite[Corollary 18.20]{Noohi-foundations} the fundamental group can be identified with the group of deck transformations of the universal cover. Since $q$ is a $\mu_n$-bundle, this shows that
\[
\pi_1(\mc{U}_1, x) = \mu_n =  \langle c \mid c^n=1 \rangle.
\]
Finally, the intersection $\mc{V}$ is isomorphic to the punctured disk, therefore
\[
\pi_1(\Vc, x) = \langle d \rangle.
\]
Now we can use the Seifert--Van Kampen theorem to compute the fundamental group of $\mc{X}$. Indeed, the inclusions $j_1\colon \mc{V} \to \mc{U}_1$ and $j_2 \colon \mc{V} \to \mc{U}_2$ induce
\begin{align}
j_{1*}\colon \pi_1(\mc{V}, x) & \to \pi_1(\mc{U}_1, x)\\
d & \mapsto c,
\end{align}
\begin{align}
j_{2*}\colon \pi_1(\mc{V}, x) & \to \pi_1(\mc{U}_2, x)\\
d & \mapsto \prod_{i = 1}^g [a_i, b_i].
\end{align}
Taking the pushout we get the above presentation of $\pi_1(\XX)$.
\end{proof}

As a consequence, we get an explicit description of the Betti moduli space of $\XX$.
\begin{cor} \label{Betti moduli space}
The Betti moduli space of the stacky curve $\XX$ can be expressed as a quotient of an affine variety as
\begin{equation}
\mc{M}_{n, B}(\XX) \cong \left[ \bigg\{ ((A_i, B_i)_{i = 1}^g, \, C) \in GL_n^{2g +1} \mymid \prod_{i=1}^{g}[A_i, B_i] = C, \, C^n = Id \bigg\} / GL_n \right]. \label{eq: betti moduli space}
\end{equation}
\end{cor}

\subsection{The monodromy representation}
We now want to understand what the component $\mc{M}_{n \cdot w(-d), DR}^0(\XX)$ corresponds to in $\mc{M}_{n, B}(\XX)$ under the Riemann-Hilbert correspondence
\begin{equation} \label{Riemann -Hilbert}
\mc{M}_{n, DR}(\mc{X})^{\text{an}} \cong \mc{M}_{n, B}(\mc{X})^{\text{an}},
\end{equation}
according to the description of $\mc{M}_{n, B}(\XX)$ given in \zcref{Betti moduli space}.

The isomorphism in \zcref{Riemann -Hilbert} is obtained by mapping a pair $(\mc{E}, \nabla)$ to the associated monodromy representation at $x$, where the pair is considered as a flat bundle with respect to the underlying analytic structure. In particular, the matrices $A_i, \, B_i, \, C$ in \eqref{eq: betti moduli space} are obtained by the action of the generators $a_i, \, b_i, \, c$ of the fundamental group of $\mc{X}$, as expressed in \zcref{prop: fundamental group of stacky curve}, and are well defined up to conjugation.

The multiplicity vector $\underline{m}(\mc{E})$ depends only on the behavior of $\mc{E}$ at the stacky point $p$, and therefore can be computed in a neighborhood of $p$ isomorphic to $[\Db / \mu_n]$. In particular, looking back at the Seifert-Van Kampen computation of $\pi_1(\mc{X}, x)$ in \zcref{prop: fundamental group of stacky curve}, we see that the decomposition \eqref{eq: decomposition of MDR} will translate on the Betti moduli space only in terms of conditions on the matrix $C$.

After restricting to $[\Db/ \mu_n]$, we can find the matrix $C$ as obtained by the monodromy action of the generator of $\pi_1([\Db / \mu_n], x) \cong \mu_n$ on $(\mc{E}\vert_{[\Db / \mu_n]}, \nabla\vert_{[\Db / \mu_n]})$.
\begin{prop} \label{prop: monodromy representation}
Let $(\mc{E}, \nabla)$ be a vector bundle with a flat connection on $[\Db / \mu_n]$. Then the representation of $\pi_1([\Db / \mu_n], x) \cong \mu_n$ obtained by the monodromy action on $(\mc{E}, \nabla)$ coincides with the representation obtained as $\mc{E}\vert_{B \mu_n}$.
\end{prop}
\begin{proof}
The statement amounts to observing that both representations come from the equivariant structure induced by the atlas
\begin{equation}
p \colon \Db \to [\Db / \mu_n].
\end{equation}
Since $\Db$ is simply connected, this is the universal bundle and the monodromy action $(\mc{E}, \nabla)$ is the $\mu_n$-action on the pullback $(p^*\mc{E}, p^*\nabla)$. This is the trivial bundle with the trivial connection, so we get
\begin{equation}
\rho\colon  \mu_n \to \Aut_{\O(\Db)}(\O(\Db)^n, \dd) = GL_n(\Cb) \subseteq GL_n(\O(\Db)).
\end{equation}
On the other side, by restricting to $B\mu_n$ and pulling back to the universal bundle, we get a vector bundle on the point $\{0\}$ with a $\mu_n$-equivariant structure, i.e. a representation
\begin{equation}
\sigma\colon \mu_n \to GL_n(\Cb).
\end{equation}
To show that $\rho = \sigma$, notice that the inclusion $\{0\} \subseteq \Db$ pulls back to the morphism of evaluation at $0$
\begin{equation}
\text{ev}_0 \colon GL_n(\O(\Db)) \to GL_n (\Cb),
\end{equation}
so that $\sigma = \text{ev}_0 \circ \rho = \rho$.
\end{proof}

\begin{cor}\label{cor: betti moduli space decomposition}
Under the identification \eqref{eq: betti moduli space}, the connected component of the Betti moduli space corresponding to the multiplicity vector $\underline{m}= m_0 \, w(0) + \dots + m_{n-1} \, w(n-1)$ is
\begin{equation}
\mc{M}_{\underline{m}, B}(\XX) \cong \left[ \bigg\{ (A_i, B_i)_{i = 1}^g \in GL_n^{2g} \mymid \prod_{i=1}^{g}[A_i, B_i] = C_{\underline{m}} \bigg\} / GL_n \right],
\end{equation}
where $C_{\underline{m}}$ is a matrix representing the action of the generator of $\mu_n$ in the representation with multiplicity vector $\underline{m}$. In particular, $C_{\underline{m}}$ can be taken as a diagonal matrix with exactly $m_l$ entries equal to $\mathrm{e}^{\frac{2 \pi \mathrm{i}}{n} l}$.
\end{cor}
In particular, the case we are interested in is $\underline{m} = n \cdot w(-d)$.
\begin{cor}
The Riemann-Hilbert correspondence \eqref{Riemann -Hilbert} maps the weight $-d$ component of the De Rham moduli space $\mc{M}_{n \cdot w(-d), DR}(\XX)$ to
\begin{equation}
\mc{M}_{n \cdot w(-d), B}(\XX) \cong \left[ \bigg\{ (A_i, B_i)_{i = 1}^g \in GL_n^{2g} \mymid \prod_{i=1}^{g}[A_i, B_i] = \mathrm{e}^{-\frac{2 \pi \mathrm{i}}{n} d} Id \bigg\} / GL_n \right].
\end{equation}
\end{cor}

This is Theorem \ref{intro thm: twisted character variety} presented in the introduction, showing that the twisted Betti moduli space of the curve $X$ is a component of the Betti moduli space of the stacky curve $\XX$.
\section{The non abelian Hodge correspondence}\label{section: The non abelian Hodge correspondence}
The non abelian Hodge correspondence, due to Hitchin, Corlette and Simpson, establishes an homeomorphism between the analytic varieties underlying the Dolbeault, De Rham and Betti moduli spaces on a smooth projective variety. The classical argument relies on differential methods, comparing the differential equations satisfied by the operators defining the various structures on a vector bundle. 
In \cite[]{simpson-local_systems_on_proper} Simpson extended the correspondence to smooth and proper Deligne-Mumford stacks by relying on the classical case and using simplicial resolutions. Namely, he showed that there are homeomorphisms
\begin{equation}
M_{Dol}(\mc{Y}) \cong M_{DR}(\mc{Y}) \cong M_{Hdg}(\mc{Y})
\end{equation}
for a smooth proper Deligne-Mumford stack $\mc{Y}$ over $\Cb$. This in particular applies to the case of the stacky curve that we are studying, inducing isomorphisms in singular cohomology
\begin{equation} \label{eq: iso cohom}
H^*(M_{Dol}(\XX), \Qb) \cong H^*(M_{DR}(\XX), \Qb) \cong H^*(M_{Hdg}(\XX), \Qb).
\end{equation}

However the isomorphisms in \zcref{eq: iso cohom} can be obtained directly by including the Dolbeault and De Rham moduli spaces in the moduli space of $\lambda$-connections on $\XX$, as we shall explain in this section. The advantage of this approach is that the isomorphisms \eqref{eq: iso cohom} are obtained algebraically, with the only analytic input being the Riemann-Hilbert correspondence. In particular, since the isomorphism
\begin{equation}
H^*(M_{Dol}(\XX), \Qb) \cong H^*(M_{DR}(\XX), \Qb)
\end{equation}
is obtained algebraically, it is actually an isomorphism of Hodge structures.

The key idea, due to Simpson, is realizing the Dolbeault and De Rham moduli spaces as the special and generic fiber inside the Hodge moduli space, i.e. the moduli space of vector bundles with $\lambda$-connections.

\begin{defi}
Let $\mc{E}$ be a vector bundle on $\XX$. For a fixed complex number $\lambda$, define a \textit{$\lambda$-connection} $\mc{E}$ to be a $\Cb$-linear morphism $\nabla\colon \mc{E} \to \mc{E} \otimes \Omega_{\XX}$, respecting the Leibniz rule up to the scalar $\lambda$, namely
\[
\nabla(f \, s) = f \, \nabla(s) + \lambda \, s \otimes \dd f
\]
for all sections $f \in \O_{\XX}(U)$ and $s \in \mc{E}(U)$.
Define the \textit{Hodge moduli space} of $\XX$ to be the moduli stack parametrizing semistable triples $(\mc{E}, \lambda, \nabla)$, where $\mc{E}$ is a vector bundle on $\XX$, $\nabla$ is a $\lambda$-connection on $\mc{E}$ and $\mc{E}$ is slope-semistable with respect to $\nabla$. Explicitly,
\begin{equation}
\mc{M}_{\underline{m}, Hdg}^0(\XX) \coloneqq \left< (\mc{E}, \lambda, \nabla) \mymid
\begin{gathered}
\mc{E}\in \Bun_{\underline{m}}^0(\XX), \; \lambda \in \Ab^1, \; \nabla\colon \mc{E} \to \mc{E}\otimes \Omega_{\XX} \text{ is a } \lambda-\text{connection}, \\
(\mc{E}, \nabla) \text{ is semistable}
\end{gathered}
\right>.
\end{equation}
\end{defi}

\begin{rmk}
The Hodge moduli space $\mc{M}_{\underline{m}, Hdg}^0(\XX)$ has the following properties.
\begin{enumerate}[label=(\roman*)]
	\item The multiplicative group $\Gb_m$ acts on $\mc{M}_{\underline{m}, Hdg}^0(\XX)$ by
	\[
	t \cdot (\Ec, \lambda, \nabla) = (\Ec, t \lambda, t \nabla).
	\]
	\item The projection
    \begin{align}
	\pi\colon \mc{M}_{\underline{m}, Hdg}^0(\XX) &\to \Ab^1 \\
	(\mc{E}, \lambda, \nabla) &\mapsto \lambda,
	\end{align}
    is $\Gb_m$-equivariant, where $\Gb_m$ acts on $\Ab^1$ with weight one.
	\item For $\lambda \neq 0$ all of the fibers $\pi^{-1}(\lambda)$ are isomorphic to $\pi^{-1}(1) = \mc{M}_{\underline{m}, DR}^0(\XX)$, while $\pi^{-1}(0)= \mc{M}_{\underline{m}, Dol}^0(\XX)$.
\end{enumerate}
\end{rmk}

Summing up, we have the following diagram
\begin{equation} \label{diagram Hodge}
    \begin{tikzcd}[ampersand replacement=\&, column sep=large, row sep=huge]
        {\mc{M}_{\underline{m}, Dol}^0(\XX)} \& {\mc{M}_{\underline{m},Hdg}^0(\XX)} \& {\mc{M}_{\underline{m}, DR}^0(\XX)} \& {\mc{M}_{\underline{m},B}(\XX)} \\
        {\{ 0 \}} \& {\Ab^1} \& {\{ 1 \},}
        \arrow["i", hook, from=1-1, to=1-2]
        \arrow[from=1-1, to=2-1]
        \arrow["\lrcorner"{anchor=center, pos=0.125}, draw=none, from=1-1, to=2-2]
        \arrow["\pi", from=1-2, to=2-2]
        \arrow["j"', hook', from=1-3, to=1-2]
        \arrow["\RH", "\cong"', from=1-3, to=1-4]
        \arrow["\lrcorner"{anchor=center, pos=0.125, rotate=-90}, draw=none, from=1-3, to=2-2]
        \arrow[from=1-3, to=2-3]
        \arrow[from=2-1, to=2-2]
        \arrow[from=2-3, to=2-2]
    \end{tikzcd}
\end{equation}
where $R\!H$ is defined analytically via the Riemann-Hilbert correspondence and $\pi$ is a trivial fibration over $\Ab^1 \setminus \{0\}$.

Restrict now to the case $\underline{m} = n \cdot w(-d)$ and suppose that $n$ and $d$ are coprime. This hypothesis ensures that a point $(\mc{E}, \lambda, \nabla)$ is semistable if and only if it is stable. Consequently, the coarse moduli space $M_{n \cdot w(-d), Hdg}^0(\XX)$ is smooth, the morphism $\pi$ is smooth and the moduli map
\begin{equation}
\mc{M}_{n \cdot w(-d), Hdg}^0(\XX) \to M_{n \cdot w(-d), Hdg}^0(\XX)
\end{equation}
is a $\Gb_m$-gerbe.

The diagram \eqref{diagram Hodge} induces diagram of smooth varieties after passing to the coarse moduli spaces
\begin{equation}\label{moculi spaces diagram}
\begin{tikzcd}[row sep=huge]
	{M_{n \cdot w(-d), Dol}^0(\XX)} & {M_{n \cdot w(-d), Hdg}^{0}(\XX)} & {M_{n \cdot w(-d), DR}^0(\XX)} & {M_{n \cdot w(-d), B}(\XX)} \\
	{\{ 0 \}} & {\Ab^1} & {\{ 1 \}.}
	\arrow["i", hook, from=1-1, to=1-2]
	\arrow[from=1-1, to=2-1]
	\arrow["\lrcorner"{anchor=center, pos=0.125}, draw=none, from=1-1, to=2-2]
	\arrow["\pi", from=1-2, to=2-2]
	\arrow["j"', hook', from=1-3, to=1-2]
	\arrow["\RH", "\cong"', from=1-3, to=1-4]
	\arrow["\lrcorner"{anchor=center, pos=0.125, rotate=-90}, draw=none, from=1-3, to=2-2]
	\arrow[from=1-3, to=2-3]
	\arrow[from=2-1, to=2-2]
	\arrow[from=2-3, to=2-2]
\end{tikzcd}
\end{equation}
To show that $i$ and $j$ induce isomorphisms in cohomology we want to use the Bia\l{y}nicki-Birula decomposition of $M_{n \cdot w(-d), Hdg}^{0}(\XX)$, similarly to what Nakajima did for point counting in the appendix to \cite{appendix-by-nakajima}.

In \cite[]{Hausel-RodriguezVillega} Hausel and Rodriguez-Villega showed for a variety to admit a Bia\l{y}nicki-Birula decomposition it doesn't need to be projective, but a weaker notion is sufficient.
\begin{defi}\label{defi: semiprojective}
A complex quasi projective variety $M$ with an action of the multiplicative group $\Gb_m$ is \textit{semiprojective} if
\begin{enumerate}[label={(\roman*)}, ref={\thedefi.(\roman*)}]
	\item \label{semiprojecive i} the fixed-point locus $M^{\Gb_m}$ is proper;
	\item \label{semiprojecive ii} for all $x \in M$ the limit point $\displaystyle \lim_{t \to 0} t \cdot x$ exists, meaning that the $\Gb_m$-equivariant morphism 
	\begin{align}
	\Gb_m &\to M \\
	 1 &\mapsto x
	\end{align}
	can be extended to $\Ab^1$.
\end{enumerate}
\end{defi}

Using \cite[Theorem 6.10]{simpson-local_systems_on_proper}, we know that $M_{n, Hdg}^0(\XX)$ is quasi projective, as is the stratum $M_{n \cdot w(-d), Hdg}^{0}(\XX)$. We want to show that it is semiprojective. \\
The property \ref{semiprojecive i} follows from the properness of the Hitchin fibration on $\Higgs_n^d(X)$.\\
For the property \ref{semiprojecive ii}, notice that if we do not require the limit point to be semistable, the existence of limit points is clear because the action scales the connection to the zero Higgs field. To get a semistable Higgs bundle, it suffices to use a semistable reduction argument, analogous to Langton's algorithm for vector bundles. This argument was used in \cite[Corollary 10.2]{Simpson-Hodge-filtration} to show the existence of limit points in the Dolbeault moduli space of a smooth projective variety, restricting to the case of vector bundles with vanishing Chern classes. Using our approach this can be shown for $\XX$ and therefore for vector bundles of non zero degree on $X$.

\begin{prop}
For any $\Cb$-point $x$ of $\mc{M}_{n \cdot w(-d),Hdg}^0(\XX)$ the limit point $\displaystyle \lim_{t \to 0} t \cdot x$ exists.
\end{prop}
\begin{proof}
Let $\mc{N}$ be the moduli stack parametrizing vector bundles on $\XX$ with a $\lambda$-connection, without any stability assumptions. Notice that over $\lambda \neq 1$ the stack $\mc{N}$ coincides with $\mc{M}_{n \cdot w(-d),Hdg}^0(\XX)$ because flat bundles are stable.
Over $\lambda = 0$ instead, $\mc{M}_{n \cdot w(-d), Dol}^0(\XX)$ has been enlarged to $\Higgs_{n \cdot w(-d)}^0(\XX)$.
The $\Gb_m$-action on $\mc{M}_{n \cdot w(-d),Hdg}^0(\XX)$ extends to an action on $\mc{N}$, for which limit points in $0$ exist. Indeed, for any $k$-point $(\mc{E}, \lambda, \nabla)$ the orbit
\begin{align}
\Gb_m &\longrightarrow \mc{N} \\
t &\mapsto (\mc{E}, t \lambda, t\nabla )
\end{align}
can be extended to $\Ab^1$ by $0 \mapsto (\mc{E}, 0, 0)$.\\
Using the equivalence $\Higgs_{n \cdot w(-d)}^0(\XX) \cong \Higgs_n^d(X)$, and the fact that the unstable locus in $\mc{N}$ is contained in $\Higgs_{n \cdot w(-d)}^0(\XX)$, we get that the Harder-Narasimhan stratification of $\Higgs_n^d(X)$ induces a $\Theta$-stratification of $\mc{N}$, whose semistable locus is exactly $\mc{M}_{n \cdot w(-d), Hdg}^0(\XX)$.
Then using the semistable reduction theorem \cite[Theorem 6.5]{semistable-reduction} we get that the existence of limit points in $\mc{N}$ implies the existence of limit points in $\mc{M}_{n \cdot w(-d), Hdg}^0(\XX)$.
\end{proof}

\begin{cor}
The coarse moduli space $M_{n \cdot w(-d), Hdg}^0(X)$ is a semiprojective variety.
\end{cor}

The fibers of a semiprojective variety with respect to a smooth $\Gb_m$-equivariant morphism 
\begin{equation}
M \to \Ab^1
\end{equation}
all have cohomology isomorphic to that of the total space. For the zero-fiber $M_0$ this is shown by comparing the Bia\l{y}nicki-Birula stratifications of the total space $M$ and of the fiber $M_0$. For the generic fiber this is shown using the Gysin sequence of the inclusion $M_0 \hookrightarrow M$ and the triviality of $M$ over $\Gb_m$. This argument is explained in detail in \cite[Theorem B.1]{motives-of-higgs-bundles}, see also \cite[Corollary 1.3.3]{Hausel-RodriguezVillega}.
Applying this result to our setting, we get the following theorem.

\begin{thm} \label{thm: isomorphic cohomologies}
The morphisms
\begin{gather}
i^*\colon H^*(M_{n \cdot w(-d), Hdg}^{0}(\XX), \Qb) \to H^*(M_{n \cdot w(-d), Dol}^0(\XX), \Qb), \\
j^*\colon  H^*(M_{n \cdot w(-d), Hdg}^{0}(\XX), \Qb) \to H^*(M_{n \cdot w(-d), DR}^0(\XX), \Qb)
\end{gather}
induced by the inclusions $i$ and $j$ are isomorphisms of pure mixed Hodge structures.
\end{thm}

Using the comparison between degree $0$ bundles on $\XX$ and degree $d$ bundles on $X$ proven in \zcref{section: Higgs bundles and Dolbeault moduli space,section: De Rham and Betti moduli spaces}, we get the following.

\begin{cor}
Composing the inclusions $i$ and $j$ with the twisted pullback and pushforward along $\XX \to X$,
we get an isomorphism of cohomology rings
\begin{equation}
H^*(M_{n, Dol}^d(X), \Qb) \cong H^*(M_{n, DR}^d(X), \Qb),
\end{equation}
inducing isomorphism of mixed Hodge structures on all the graded pieces. Moreover, composing the Riemann-Hilbert correspondence on $\XX$ with the twisted pullback functor we get an isomorphism
\begin{equation}
H^*(M_{n, DR}^d(X), \Qb) \cong H^*(M_{n, B}^d(X), \Qb),
\end{equation}
where $M_{n, B}^d(X)$ is the \textit{twisted character variety}
\begin{equation}
\bigg\{ (A_i, B_i)_{i = 1}^g \in GL_n^{2g} \;\bigg|\; \prod_{i=1}^{g}[A_i, B_i] = \mathrm{e}^{-\frac{2 \pi \mathrm{i}}{n} d} Id \bigg\} \modd GL_n.
\end{equation}
\end{cor}

This corollary can be understood as a cohomological version of the non abelian Hodge correspondence for vector bundles of degree $d$ on $X$.
\section{The Hodge weights}\label{section: The Hodge weights}
In the previous section we described isomorphisms in rational cohomology
\begin{gather}
i^*\colon H^*(M_{n \cdot w(-d), Hdg}^{0}(\XX)) \to H^*(M_{n \cdot w(-d), Dol}^0(\XX)),\\
j^*\colon H^*(M_{n \cdot w(-d), Hdg}^{0}(\XX)) \to H^*(M_{n \cdot w(-d), DR}^0(\XX)),\\
\RH^*\colon H^*(M_{n \cdot w(-d), B}(\XX)) \to H^*(M_{n \cdot w(-d), DR}^0(\XX)).
\end{gather}
These are different in nature, since the morphisms $i$ and $j$ are algebraic, while $\RH$ is only defined analytically. Consequently, while $i^*$ and $j^*$ are isomorphisms of pure mixed Hodge structures, $\RH^*$ doesn't preserve the Hodge structure.

In order to understand how $\RH^*$ acts on the Hodge weights, we first describe generators of the cohomology ring $H^*(M_{n \cdot w(-d), DR}^0(\XX))$ with their weight.
As already observed, we are working with moduli of stable bundles, therefore for all of the stacks considered the coarse moduli map
\begin{equation}
\mc{M} \to M
\end{equation}
is a $\Gb_m$-gerbe. Consequently \cite[Lemma 3.1]{Heinloth-Gm-gerbes} shows that
\begin{equation}
H^*(\mc{M}, \Qb) \cong H^*(M, \Qb)[z],
\end{equation}
where $z$ is the first Chern class of a vector bundle on $\mc{M}$ of weight 1. In particular, studying $\RH^*$ at the level of the stacks or of the coarse moduli spaces is equivalent and in what follows we will alternate between the two.

Beauville showed that the cohomology ring of the moduli space of vector bundles over a curve is generated by the tautological classes of the universal bundle. The argument was then generalized to moduli of Higgs bundles by Markman.\\
Let $(\mc{E}_{\text{univ}}, \lambda_{\text{univ}}, \nabla_{\text{univ}})$ be the universal bundle on $\XX \times \mc{M}_{n \cdot w(-d), Hdg}^0(\XX)$. Pulling back along the inclusions $i$ and $j$, the pair $(\mc{E}_{\text{univ}}, \nabla_{\text{univ}})$ restricts to the universal Higgs bundle $(\mc{E}_{\text{univ}}^{Dol}, \varphi_{\text{univ}})$ and the universal flat bundle $(\mc{E}_{\text{univ}}^{DR}, \nabla_{\text{univ}})$. Therefore, the induced cohomology isomorphism
\begin{equation}
j^* \circ (i^*)^{-1} \colon H^*(M_{n \cdot w(-d), Dol}^0(\XX) \times \XX) \to H^*(M_{n \cdot w(-d), DR}^0(\XX) \times \mc{X}).
\end{equation}
maps the Chern classes of $\mc{E}_{\text{univ}}^{Dol}$ to the Chern classes of $\mc{E}_{\text{univ}}^{DR}$.

Markman showed that the Künneth components of these Chern classes form generators for the cohomology rings.
The following theorem is \cite[Theorem 3]{MARKMAN2006}, see also \cites{markman2001}{Tamas-Hausel}.
\begin{thm}[Markman]
Let $c_i(\mc{E}_{\text{univ}}) \in H^{2i}(M_{n, Dol}^d(X) \times X; \Zb)$ be the $i$-th Chern class of the universal bundle. Write its Künneth decomposition as
\begin{equation}
c_i(\mc{E}_{\text{univ}}) = \sum_{j+k = 2i} e_{i,j} \otimes x_{i,k}.
\end{equation}
Then the Künneth components $e_{i,j}$ generate the cohomology ring of $M_{n, Dol}^d(X)$.
\end{thm}

Using the isomorphism $M_{n, Dol}^d(X) \cong M_{n \cdot w(-d)}^0(\XX)$ shown in \zcref{thm: equivalence for Higgs}, we get that the Künneth components of $c_i(\mc{E}_{\text{univ}}^{Dol})$ generate the cohomology ring $H^*(M_{n \cdot w(-d), Dol}^0(\XX))$. Consequently, the Künneth components of $c_i(\mc{E}_{\text{univ}}^{DR})$ generate the cohomology ring $H^*(M_{n \cdot w(-d), DR}^0(\XX))$.

Since moreover the cohomology of $M_{n, Dol}^d(X)$ carries a pure Hodge structure (see \cite[]{mehta2002hodgestructurecohomologymoduli}), as clearly does $X$, the classes $c_i(\mc{E}_{\text{univ}})$ have pure weight $2i$, while the K\"unneth components $e_{i,j}$ have pure weight $j$.

In conclusion, we have that the cohomology ring of $M_{n \cdot w(-d), Dol}^0(\XX)$ (or equivalently the cohomology ring of $M_{n \cdot w(-d), DR}^0(\XX)$) is generated by the Künneth components of the Chern classes of the universal bundle. The cohomology carries a pure Hodge structure and therefore the component $e_{i,j} \in H^j(M_{n \cdot w(-d), Dol}^0(\XX))$ has pure weight $j$.

Finally, consider
\begin{equation}
\RH^* \colon H^*(M_{n \cdot w(-d), DR}^0(\XX)) \to H^*(M_{n \cdot w(-d), B}(\XX)).
\end{equation}
This induces an isomorphism on the cohomology rings, therefore the classes $\RH^* e_{i,j}$ generate the cohomology ring of $M_{n \cdot w(-d), B}(\XX)$. But it is not a morphism of mixed Hodge structures, so it needn't preserve the weights.
In \cite[]{Shende-weights} Shende computed the Hodge weights of the tautological classes $\RH^* e_{i,j}$ in the case of the character variety of a topological space.

\begin{thm}[Shende] \label{thm: shende weights}
Let $Y$ be a complex manifold and $M_B(Y)$ its character variety. Then the tautological classes on $H^*(M_B(Y))$ coming from the Chern classes {$c_i(\mc{E}_{\text{univ}}) \in H^{2i}(M_B(Y) \times Y)$} have pure weight $2i$.
\end{thm}

Shende then generalized this theorem to the twisted character variety, by reducing it to the case of $GL_n$- and $PGL_n$-character varieties, as explained in \cite[]{MixedHodgePolynomials}.
In what follows, we shall explain how \zcref{thm: shende weights} can be directly generalized to our setting, without needing to use \cite[]{MixedHodgePolynomials}.

The main step appearing in Shende's proof of \zcref{thm: shende weights} is an equivalence
\begin{equation}
[\Hom(\pi_1(Y), G)/G] \cong \Hom_{SSch}(\Delta_Y, BG),
\end{equation}
where $Y$ is a path-connected topological space and $\Delta_Y$ is a simplicial set with geometric realization homotopy equivalent to $Y$, made into a simplicial scheme by replacing every simplex with $\Spec \Cb$.
Let us explain this equivalence in more detail. We first work with simplicial sets and groupoids in sets and denote by $BG$ both the groupoid and the simplicial set constructed from a group $G$.

First of all notice that if $H$ and $G$ are groups, there is an equivalence of groupoids
\begin{equation} \label{eq: eqv groupoids}
[\Hom(H, G)/G] \cong \Hom_{Gpd}(BH, BG).
\end{equation}
In the case $H = \pi_1(Y)$, we have that if $Y$ is path connected then the fundamental group is equivalent to the fundamental groupoid
\begin{equation}
B\pi_1(Y) \cong \Pi_1(Y).
\end{equation}

Let $\Delta_Y$ be a simplicial set such that $\pi_1(\Delta_Y) \cong \pi_1(Y)$. Then the fundamental groupoid can be realized simplicially as the free groupoid constructed from the path category of $\Delta_Y$. Denote by $P_* \Delta_Y$ the \textit{path category} of $\Delta_Y$, i.e. the category with objects the vertices of $\Delta_Y$ and morphisms generated by $1$-simplices in $\Delta_Y$ modulo the relation $\partial_0 \sigma \, \partial_2 \sigma = \partial_1 \sigma$ for all $2$-simplices $\sigma$. Then the fundamental groupoid of $Y$ is equivalent to the free groupoid $G P_* \Delta_Y$ on the path category $P_* \Delta_Y$, see \cite[\S III.1]{SimplicialHomotopyTheory}.\\
Substituting in \zcref{eq: eqv groupoids} we get an equivalence
\begin{equation} \label{eq: eqv groupoids 2}
[\Hom(\pi_1(Y), G)/G] \cong \Hom_{Gpd}(GP_*\Delta_Y, BG).
\end{equation}
Recall that the free groupoid functor is left adjoint to the inclusion $Gpd \hookrightarrow Cat$, so we get
\begin{equation}
\Hom_{Gpd}(GP_*\Delta_Y, BG) \cong \Hom_{Cat}(P_*\Delta_Y, BG).
\end{equation}
Finally, as explained in \cite[\S III.1]{SimplicialHomotopyTheory}, notice that since the simplicial set $BG$ is a $2$-coskeleton, a morphism $\Delta_Y \to BG$ is uniquely determined by its value on $0$- and $1$-simplices, up to compatibility with boundary of $2$-simplices, so we get
\begin{equation}\label{eq: eqv groupoids 3}
\Hom_{Cat}(P_*\Delta_Y, BG) \cong \Hom_{SSet}(\Delta_Y, BG).
\end{equation}
Summing up, we get an equivalence groupoids
\begin{equation} \label{eq: eqv groupoids 4}
[\Hom(\pi_1(Y), G)/G] \cong \Hom_{SSet}(\Delta_Y, BG),
\end{equation}
where the right-hand side is a groupoid under simplicial homotopy.

Going back to the algebraic case, the above construction carries through if one considers $\pi_1(Y)$ as a discrete group over $\Cb$ and views $\Delta_Y$ as a constant simplicial scheme by replacing every simplex with $\Spec \Cb$.
Explicitly, suppose that $\pi_1(Y)$ is finitely presented, so that $\Hom(\pi_1(Y), G)$ is an algebraic variety. Consider the presheaves of groupoids
\begin{align}
[\Hom(\pi_1(Y), G)/G]^{\text{pre}}\colon Sch/\Cb &\to Gpd\\
T &\mapsto [\Hom(\pi_1(Y), G)(T)/G(T)]
\end{align}
and
\begin{align} \label{eq: Hom_SSch as presheaf}
\Hom_{SSch}(\Delta_Y, BG)^{\text{pre}}\colon Sch/\Cb &\to Gpd\\
T &\mapsto \Hom_{SSch}(\Delta_{Y,T}, BG).
\end{align}
Then the equivalences \eqref{eq: eqv groupoids}, \eqref{eq: eqv groupoids 2}, \eqref{eq: eqv groupoids 3} become
\begin{multline}
[\Hom(\pi_1(Y), G)(T)/G(T)] \cong \Hom_{Gpd}(B\pi_1(Y)_T, BG)\\
\cong \Hom_{Gpd}(\Pi_1(Y)_T, BG) \cong \Hom_{SSch}(\Delta_{Y,T}, BG).
\end{multline}
After sheafification we get the following result.
\begin{lemma}\label{lemma: simplicial carachter variety}
Let $\Delta_Y$ be a path-connected simplicial set with finitely presented fundamental group $\pi_1(Y)$. Consider $\pi_1(Y)$ as a discrete group over $\Cb$ and $\Delta_Y$ as a constant simplicial scheme, replacing every simplex with $\Spec \Cb$. Then there is an equivalence of algebraic stacks
\begin{equation} \label{eq: eqv of groupoids, stacks}
[\Hom(\pi_1(Y), G)/G] \cong \Hom_{SSch}(\Delta_Y, BG),
\end{equation}
where $\Hom_{SSch}(\Delta_Y, BG)$ is the stack obtained as stackification of the presheaf of groupoids $\Hom_{SSch}(\Delta_Y, BG)^{\text{pre}}$ described in \eqref{eq: Hom_SSch as presheaf}.
\end{lemma}

We now want to apply \zcref{lemma: simplicial carachter variety} to the study of the Betti moduli space of the stacky curve $\XX$. The main step is the construction of a triangulation of $\XX$, namely a simplicial set $\Delta_{\XX}$ such that $\pi_1(\Delta_{\XX}) \cong \pi_1(\XX)$.

Considering $X$ as a complex manifold, the stack $\XX$ can be realized topologically as a pushout
\begin{equation} \label{eq: X as pushout topologically}
\begin{tikzcd}
	{\Sb^1} & {X \setminus \Db} \\
	{\left[\overline{\Db}/\mu_n\right]} & {\mc{X},}
	\arrow[from=1-1, to=1-2]
	\arrow[from=1-1, to=2-1]
	\arrow[from=1-2, to=2-2]
	\arrow[from=2-1, to=2-2]
	\arrow["\lrcorner"{anchor=center, pos=0.125, rotate=180}, draw=none, from=2-2, to=1-1]
\end{tikzcd}
\end{equation}
where $\overline{\Db }\subseteq \Cb$ is the closed unitary disk with interior $\Db$ and boundary $\Sb^1$. Of course, the topological spaces $\Sb^1$ and $X \setminus \left\{ p \right\}$ are easily realized as simplicial sets, so we need to construct a simplicial version of the stacky quotient $[\overline{\Db} / \mu_n]$ and then get $\mc{X}$ as pushout.

In order to construct a simplicial version of $[\overline{\Db} / \mu_n]$, we start from a triangulation of the disk $\overline{\Db}$ in $3n$ triangles, compatible with the $\mu_n$-action, as pictured below.
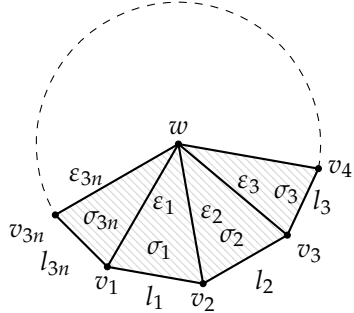
\begin{equation}
\begin{tikzpicture}[scale=0.75, >=stealth]
  \coordinate (w) at (0, 0);
  \coordinate (v_e_start) at (210:2.5);
  \coordinate (v_1) at (240:2.5);
  \coordinate (v_2) at (280:2.5);
  \coordinate (v_3) at (320:2.5);
  \coordinate (v_4) at (350:2.5);

  \draw[dashed, black] (350:2.5) arc (350:570:2.5);

  \foreach \pA/\pB in {v_e_start/v_1, v_1/v_2, v_2/v_3, v_3/v_4} {
    \fill[pattern=north west lines, pattern color=gray!30] 
      (w) -- (\pA) -- (\pB) -- cycle;
  }

  \draw[thick] (w) -- (v_e_start) node[midway, left, xshift=-1pt, yshift=1pt] {$\varepsilon_{3n}$};
  \draw[thick] (w) -- (v_1) node[midway, right] {$\varepsilon_1$};
  \draw[thick] (w) -- (v_2) node[midway, right] {$\varepsilon_2$};
  \draw[thick] (w) -- (v_3) node[midway, right, xshift=-2pt, yshift=2pt] {$\varepsilon_3$};
  \draw[thick] (w) -- (v_4);

  \draw[thick] (v_e_start) -- (v_1) node[midway, below left] {$l_{3n}$};
  \draw[thick] (v_1) -- (v_2) node[midway, below] {$l_1$};
  \draw[thick] (v_2) -- (v_3) node[midway, below right] {$l_2$};
  \draw[thick] (v_3) -- (v_4) node[midway, right] {$l_3$};

  \foreach \p in {w, v_e_start, v_1, v_2, v_3, v_4} {
    \filldraw[fill=black, draw=black] (\p) circle (1.5pt);
  }

  \node[above] at (w) {$w$};
  \node[below left] at (v_e_start) {$v_{3n}$};
  \node[below] at (v_1) {$v_1$};
  \node[below] at (v_2) {$v_2$};
  \node[below right] at (v_3) {$v_3$};
  \node[right] at (v_4) {$v_4$};

  \node at (225:1.9) {$\sigma_{3n}$};
  \node at (260:1.9) {$\sigma_1$};
  \node at (300:1.9) {$\sigma_2$};
  \node at (335:2.1) {$\sigma_3$};

\end{tikzpicture}
\end{equation}
This gives a simplicial complex $D$ with a $\mu_n$-action, where the generator acts by
\begin{equation}
w \mapsto w, \quad v_i \mapsto v_{i+3}, \quad \varepsilon_i \mapsto \varepsilon_{i+3}, \quad \sigma_{i} \mapsto \sigma_{i+3}.
\end{equation}
We can make $D$ functorially into a simplicial set by taking the nerve of the face poset of $D$, which topologically corresponds to taking the barycentric subdivision of the chosen triangulation. This yields a simplicial set $\Delta_D$ with a $\mu_n$-action. Then we can realize the quotient $[\overline{\Db} / \mu_n]$ as the homotopy quotient
\begin{equation}
\Delta_{[D/\mu_n]} \coloneqq \holim_{\longrightarrow \mu_n} \Delta_D.
\end{equation}

\begin{lemma} \label{fundamental group of homotopy quotient}
The simplicial set $\Delta_{[D / \mu_n]}$ has fundamental group $\mu_n$.
\end{lemma}
\begin{proof}
Using \cite[\S IV Lemma 5.7]{SimplicialHomotopyTheory} we find that the following square is homotopy cartesian
\begin{equation}
\begin{tikzcd}
	{\Delta_D} & {\Delta_{[D / \mu_n]}} \\
	{*} & {B\mu_n.}
	\arrow[from=1-1, to=1-2]
	\arrow[from=1-1, to=2-1]
	\arrow["p", from=1-2, to=2-2]
	\arrow[from=2-1, to=2-2]
\end{tikzcd}
\end{equation}
This means that $\Delta_D$ and $\Delta_{[D / \mu_n]}$ can be replaced up to weak equivalence by Kan complexes making $p$ into a Kan fibration $p'$ with fiber weakly equivalent to $\Delta_D$. To conclude, use the long exact homotopy sequence of the Kan fibration $p'$, see \cite[\S I Lemma 7.3]{SimplicialHomotopyTheory}, and the fact that $\Delta_D$ is contractible.
\end{proof}

In order to proceed to realize the simplicial version of the diagram \eqref{eq: X as pushout topologically}, we need to realize $\Sb^1$ as boundary of $\Delta_{[D/\mu_n]}$. Start again from the simplicial complex $D$, and consider its boundary $\partial D$.
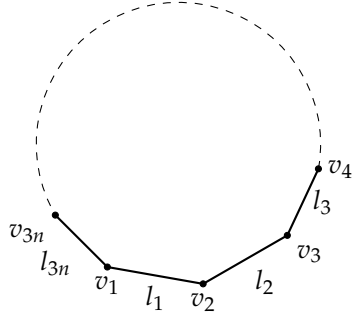
\begin{equation}
\begin{tikzpicture}[scale=0.75, >=stealth]
  \coordinate (v_e_start) at (210:2.5);
  \coordinate (v_1) at (240:2.5);
  \coordinate (v_2) at (280:2.5);
  \coordinate (v_3) at (320:2.5);
  \coordinate (v_4) at (350:2.5);

  \draw[dashed, black] (350:2.5) arc (350:570:2.5);

  \draw[thick] (v_e_start) -- (v_1) node[midway, below left] {$l_{3n}$};
  \draw[thick] (v_1) -- (v_2) node[midway, below] {$l_1$};
  \draw[thick] (v_2) -- (v_3) node[midway, below right] {$l_2$};
  \draw[thick] (v_3) -- (v_4) node[midway, right] {$l_3$};

  \foreach \p in {v_e_start, v_1, v_2, v_3, v_4} {
    \filldraw[fill=black, draw=black] (\p) circle (1.5pt);
  }

  \node[below left] at (v_e_start) {$v_{3n}$};
  \node[below] at (v_1) {$v_1$};
  \node[below] at (v_2) {$v_2$};
  \node[below right] at (v_3) {$v_3$};
  \node[right] at (v_4) {$v_4$};
\end{tikzpicture}
\end{equation}
The action of $\mu_n$ on $\partial D$ is free and the quotient corresponds to the following triangulation of $\Sb^1$.
\begin{equation}\label{eq: dD/mue simplicial complex}
\begin{tikzpicture}[
    scale=1,
    dot/.style={circle, fill=black, inner sep=1.5pt}
  ]
    \coordinate (v1) at (0, 0);
    \coordinate (v2) at (1.5, 0);
    \coordinate (v3) at (0.75, 1.3);

    \draw (v1) -- (v2) node[midway, below=2pt] {$[l_1]$};
    \draw (v2) -- (v3) node[midway, right=2pt] {$[l_2]$};
    \draw (v3) -- (v1) node[midway, left=2pt] {$[l_3]$};

    \node[dot] at (v1) {};
    \node[dot] at (v2) {};
    \node[dot] at (v3) {};

    \node[below left=0pt] at (v1) {$[v_1]$};
    \node[below right=0pt] at (v2) {$[v_2]$};
    \node[above=2pt] at (v3) {$[v_3]$};

\end{tikzpicture}
\end{equation}
Once again, we pass functorially to simplicial sets by taking the nerve of the face poset. Then we get the boundary subsimplicial set $\Delta_{\partial D} \subseteq \Delta_D$, which has a free $\mu_n$-action such that the quotient $\Delta_{\partial D}/ \mu_n$ corresponds to the barycentric subdivision of \eqref{eq: dD/mue simplicial complex}.
\begin{rmk}
Since $\mu_n$ acts freely on $\Delta_{\partial D}$, the natural map
\begin{equation}
\Delta_{[\partial D/\mu_n]} \to \Delta_{\partial D}/\mu_n \cong \Delta_{\Sb^1}
\end{equation}
from the homotopy quotient to the strict quotient is a weak homotopy equivalence.
\end{rmk}
Indeed, the previous remark follows from the following more general statement.
\begin{lemma}
Let a group $G$ act on a simplicial set $S$. Denote by $\displaystyle [S/G]\coloneqq \holim_{\longrightarrow G} S$ the homotopy quotient and by $S/G$ the strict quotient. If the action of $G$ on $S$ is free, then the natural map
\begin{equation}
[S/G] \to S/G
\end{equation}
is a weak homotopy equivalence.
\end{lemma}
\begin{proof}
Let $\SSet_G$ be the category of simplicial sets with a $G$-action. Thinking of $G$ as a constant simplicial group, we can use \cite[\S V.2]{SimplicialHomotopyTheory} to study the closed model structure on $\SSet_G$ with weak equivalences and fibrations given by morphisms which are weak equivalences and fibrations in $\SSet$.
Then \cite[\S V Cor 2.10]{SimplicialHomotopyTheory} states that $S$ is cofibrant in $\SSet_G$ if and only the action of $G$ is free.

On the other hand, $\SSet_G$ can be thought of as the functor category $[BG, \SSet]$. Then the one specified above is exactly the projective model structure on the functor category $[BG, \SSet]$. The cofibrant objects with respect to this model structure are usually called projectively cofibrant.
With this notation, \cite[\S V, Cor 2.10]{SimplicialHomotopyTheory} says that $S$ is projectively cofibrant if and only if the action of $G$ on $S$ is free. Finally, we conclude the lemma using \cite[\href{https://kerodon.net/tag/03CJ}{Tag 03CJ}]{kerodon}, which states that if a diagram in a functor category $F \in [I, \SSet]$ is projectively cofibrant, then the natural map
\begin{equation}
\holim_{\longrightarrow I} F \to \lim_{\longrightarrow I} F
\end{equation}
is a weak equivalence.
\end{proof}

To realize the picture described in \eqref{eq: X as pushout topologically}, the last piece that we need is a simplicial model of $X \setminus \Db$. As before, to do this we take a triangulation of $X \setminus \Db$ such that the boundary $\Sb^1$ is triangulated as in \eqref{eq: dD/mue simplicial complex}. Call $Y$ the resulting simplicial complex and $\Delta_Y$ the corresponding simplicial set. By construction we have
\begin{equation}
\Delta_{\partial D/ \mu_n} = \Delta_{\partial Y} \subseteq \Delta_Y.
\end{equation}
Then the simplicial set realizing the stacky curve $\mc{X}$ is obtained as the pushout
\begin{equation} \label{eq: X as simplicial set}
\begin{tikzcd}
	{\Delta_{[\partial D / \mu_n]}} & {\Delta_{\partial D / \mu_n} = \Delta_{\partial Y}} & {\Delta_Y} \\
	{\Delta_{[D / \mu_n]}} && {\Delta_{\mc{X}}}
	\arrow["\simeq", from=1-1, to=1-2]
	\arrow["i", hook, from=1-1, to=2-1]
	\arrow["j", hook, from=1-2, to=1-3]
	\arrow[from=1-3, to=2-3]
	\arrow[from=2-1, to=2-3]
	\arrow["\lrcorner"{anchor=center, pos=0.125, rotate=180}, draw=none, from=2-3, to=1-2]
\end{tikzcd}
\end{equation}
Finally, we check that $\Delta_{\mc{X}}$ has the expected fundamental group.
\begin{prop}
The pushout diagram \eqref{eq: X as simplicial set} induces an isomorphism
\begin{equation}
\pi_1(\Delta_{\mc{X}}) \cong \left\langle a_i, \, b_i, \, c \mymid \prod_{i = 1}^g [a_i, b_i] = c, \; c^n = 1 \right\rangle .
\end{equation}
\end{prop}
\begin{proof}
Using Seifert-Van Kampen theorem for simplicial sets, see \cite[\S III, Theorem 1.4]{SimplicialHomotopyTheory}, we get the following pushout diagram of groups
\begin{equation}
\begin{tikzcd}
	{\pi_1(\Delta_{[\partial D / \mu_n]}) \cong \pi_1(\Sb^1)} & {} & {\pi_1(\Delta_Y) \cong \pi_1(X \setminus \Db)} \\
	{\pi_1(\Delta_{[D / \mu_n]}) = \mu_n} && {\pi_1(\Delta_{\mc{X}}).}
	\arrow["{j_*}", from=1-1, to=1-3]
	\arrow["{i_*}", from=1-1, to=2-1]
	\arrow[from=1-3, to=2-3]
	\arrow[from=2-1, to=2-3]
	\arrow["\lrcorner"{anchor=center, pos=0.125, rotate=180}, draw=none, from=2-3, to=1-2]
\end{tikzcd}
\end{equation}
Recall that
\begin{gather}
\pi_1(\Sb^1)=\left< d \right>, \quad \pi_1(\Delta_{[D / \mu_n]}) = \left< c \mid c^n=1 \right>,\\
\pi_1(X \setminus \Db) = \left< a_i, b_i \mid i=1,\dots,g \right>.
\end{gather}
We have that
\begin{equation}
j_* d = \prod_{i=1}^{g} [a_i, b_i]
\end{equation}
by construction of $j$. Finally notice that, under the isomorphism $\pi_1(\Delta_{[D/ \mu_n]}) \cong \pi_1(B \mu_n)$ shown in \zcref{fundamental group of homotopy quotient}, the map $i_*$ corresponds to the pushforward under the projection
\begin{equation}
\Delta_{[\partial D / \mu_n]} \to B \mu_n.
\end{equation}
Therefore, $i_* d = c$, concluding the proof.
\end{proof}

In conclusion, we constructed a simplicial set $\Delta_{\mc{X}}$ realizing a triangulation of the stacky curve $\mc{X}$, with the main property that
\begin{equation}
\pi_1(\mc{X}) \cong \pi_{1}(\Delta_{\mc{X}}).
\end{equation}
Substituting in \zcref{lemma: simplicial carachter variety}, we get the following result.
\begin{prop}\label{prop: simplicial carachter variety}
As in \zcref{lemma: simplicial carachter variety}, denote by $\Delta_{\XX}$ the constant simplicial scheme constructed from the simplicial set $\Delta_{\XX}$ and by $\Hom_{SSch}(\Delta_{\XX}, BG)$ the stack obtained by stackifying the presheaf of groupoids $\Hom_{SSch}(\Delta_{\XX}, BG)^{\text{pre}}$ described in \eqref{eq: Hom_SSch as presheaf}.
Then there is an isomorphism of stacks
\begin{equation}
\left[ \text{Hom}(\pi_1(\mc{X}, v_0), G) / G \right] \cong \Hom_{SSch}(\Delta_{\XX}, BG)
\end{equation}
between the stack of local systems on $\mc{X}$ and the stack of simplicial morphisms of $\Delta_{\mc{X}}$ into the classifying stack $BG$.
\end{prop}

Using this isomorphism, we can reformulate Shende's proof of \zcref{thm: shende weights} to compute the weight of the tautological classes on the Betti moduli space of $\XX.$
\begin{thm}
Consider the Riemann-Hilbert isomorphism
\begin{equation}
\RH^* \colon H^*(\mc{M}_{n, DR}^0(\XX)) \to H^*(\mc{M}_{n, B}(\XX)).
\end{equation}
Let $e_{i,j}$ be the tautological classes in $H^*(\mc{M}_{n, DR}^0(\XX))$. Then for all $j$ the pullback $\RH^* e_{i,j}$ has weight $2i$.
\end{thm}
\begin{proof}
Fix $G = GL_n$, so that
\begin{equation}
\mc{M}_{n, B}(\XX) \coloneqq \left[ \text{Hom}(\pi_1(\XX, v_0), GL_n) / GL_n \right]
\end{equation}
is the Betti moduli stack on $\XX$. Let
\begin{equation}
\RH \colon \mc{M}_{n, B}(\XX) \to \mc{M}_{n, DR}^0(\XX)
\end{equation}
be the analytic isomorphism induced by the Riemann-Hilbert correspondence.
As before, let $e_{i,j}$ be the K\"unneth components of the Chern classes of the universal bundle on $\mc{M}_{n,DR}^0(\XX)\times \XX$. We want to compute the weight of the classes
\begin{equation}
\varepsilon_{i,j} \coloneqq \RH^*(e_{i,j})
\end{equation}
in the cohomology ring $H^*(\mc{M}_{n,B}(\XX), \Qb)$. The universal flat bundle $\mc{E}_{\text{univ}}^{DR}$ on $\mc{M}_{n,DR}^0(\XX) \times \XX$ pulls back under $\RH$ to the flat bundle $\mc{E}_{\text{univ}}^{B}$ on $\mc{M}_{n,B}(\XX) \times \XX$, whose horizontal sections form the universal local system $L_{\text{univ}}$. We remark that $\RH$ is only defined analytically, as is the holomorphic vector bundle $\mc{E}_{\text{univ}}^B$. The classes $\varepsilon_{i,j}$ are exactly the Künneth components of $c_i(\mc{E}_{\text{univ}}^B)$.

As shown in \cite[\S 4]{Olsson-sheaves-on-artin-stacks}, any algebraic stack can be viewed as a simplicial algebraic space by taking iterated fibered products of an atlas and this construction induces an isomorphism in cohomology, preserving mixed Hodge structures. In particular, the weights of the classes $\varepsilon_{i,j}$ can be computed after replacing $\mc{M}_{n, B}(\XX)$ by a simplicial scheme.
Using \zcref{prop: simplicial carachter variety}, we have that
\begin{equation}
\mc{M}_{n, B}(\XX) \cong \Hom_{SSch}(\Delta_\XX, BGL_n),
\end{equation}
namely the Betti moduli stack is isomorphic to the stack of simplicial maps of $\Delta_{\XX}$ into $BGL_n$.
Therefore a simplicial scheme realizing $\mc{M}_{n,B}(\XX)$ is the internal $\inHom_{SSch}(\Delta_\XX, BGL_n)$. We have a universal morphism
\begin{equation}
\varphi_{\text{univ}}\colon \inHom_{SSch}(\Delta_\XX, BGL_n) \times \Delta_\XX \to BGL_n,
\end{equation}
namely the object in
\begin{multline}
\Hom_{SSch}(\inHom_{SSch}(\Delta_\XX, BGL_n) \times \Delta_\XX, BGL_n) \\
\cong \Hom_{SSch}(\inHom_{SSch}(\Delta_\XX, BGL_n), \inHom_{SSch}(\Delta_\XX, BGL_n))
\end{multline}
corresponding to the identity. This is an algebraic morphism and the corresponding vector bundle $\mc{E}_{\text{univ}}$ is an algebraic vector bundle on $\inHom_{SSch}(\Delta_{\XX}, BGL_n) \times \Delta_{\XX}$, which corresponds to $\mc{E}_{\text{univ}}^B$ after replacing $\XX$ by $\Delta_\XX$ and $\mc{M}_{n,B}(\XX)$ by $\inHom_{SSch}(\Delta_\XX, BGL_n)$.

Since the morphism $\varphi_{\text{univ}}$ is algebraic, the pullback
\begin{equation}
\varphi_{\text{univ}}^* \colon H^*(BGL_n, \Qb) \to H^*(\inHom_{SSch}(\Delta_\XX, BGL_n), \Qb) \otimes H^*(\Delta_\XX, \Qb)
\end{equation}
preserves the mixed Hodge structures. Recall that
\begin{equation}
H^*(BGL_n, \Qb) = \Qb[c_1, \dots, c_n]
\end{equation}
and that the Chern classes are exactly the pullbacks of the $c_i$'s, namely
\begin{equation}
c_i(\mc{E}_{\text{univ}}) = \varphi_{\text{univ}}^*(c_i).
\end{equation}
Since $c_i$ has pure weight $2i$, we get that
\begin{equation}
c_i(\mc{E}_{\text{univ}}) = \sum_{j} \varepsilon_{i,j} \otimes x_{i,j}
\end{equation}
has pure weight $2i$. Since the simplicial scheme $\Delta_\XX$ is formed by several copies of $\Spec \Cb$ in each degree, its mixed Hodge structure only has weight zero. Therefore the $x_{i,j}$'s have weight zero and $\varepsilon_{i,j}$ has pure weight $2i$ for all $j$.
\end{proof}

Restricting to the $n \cdot w(-d)$-stratum we get the following corollary.

\begin{cor}
Under the Riemann-Hilbert isomorphism
\begin{equation}
\RH^* \colon H^*(M_{n \cdot w(-d), DR}^0(\XX)) \to H^*(M_{n \cdot w(-d), B}(\XX))
\end{equation}
the tautological classes $e_{i,j}$ in $H^*(M_{n \cdot w(-d), DR}^0(\XX))$, which generate the cohomology ring and have weight $j$, are mapped to cohomology classes of weight $2i$.
\end{cor}

Finally, we can compose with the isomorphisms between bundles of degree $0$ on $\XX$ and bundles of degree $d$ on $X$ proven in \zcref{section: Higgs bundles and Dolbeault moduli space, section: De Rham and Betti moduli spaces} to get an alternative proof of Shende's computation of the Hodge weights for twisted character varieties.

\begin{cor}
The non abelian Hodge isomorphism for degree $d$ vector bundles on $X$
\begin{equation}
H^*(M_{n, Dol}^d(X)) \xrightarrow{\cong} H^*(M_{n,B}^d(X))
\end{equation}
maps the tautological classes $e_{i,j}$ on $H^*(M_{n, Dol}^d(X))$, which have weight $j$ and generate the cohomology ring, to classes of weight $2i$.
\end{cor}

We remark that the previous result was already present in \cite[]{Shende-weights}. There it was obtained by using both the $GL_n$ and the $PGL_n$-character varieties to deduce the result about the twisted character variety, following \cite[]{MixedHodgePolynomials}. However, our approach avoids going through the $PGL_n$ case by viewing the twist as monodromy in $\XX$.

\appendix

\section{Seifert--Van Kampen Theorem}\label{appendix: Seifert--Van Kampen Theorem}

\begin{thm}\label{van kampen}
Let $\XX$ be a topological stack and suppose that it is connected, locally path connected and semilocally $1$-connected. Let $\mc{U}_1 \xhookrightarrow{i_1} \mc{X}$ and $\mc{U}_2 \xhookrightarrow{i_2} \mc{X}$ be open substacks such that $\mc{X} = \mc{U}_1 \cup \mc{U}_2$. Let $\mc{V} \coloneqq \mc{U}_1 \times_{\mc{X}} \mc{U}_2$. Suppose that $\mc{U}_1, \mc{U}_2$ and $\mc{V}$ are connected. Then, for $x \in \mc{V}$, there is an isomorphism
\[
\pi_1(\mc{X}, x) \cong \pi_1(\mc{U}_1, x) *_{\pi_1(\mc{V}, x)} \pi_1(\mc{U}_2, x).
\]
Equivalently, the following diagram is cocartesian
\[\begin{tikzcd}
	{{\pi_1(\mc{V}, x)}} & {\pi_1(\mc{U}_1, x)} \\
	{\pi_1(\mc{U}_2, x)} & {\pi_1(\mc{X}, x).}
	\arrow[from=1-1, to=1-2]
	\arrow[from=1-1, to=2-1]
	\arrow[from=1-2, to=2-2]
	\arrow[from=2-1, to=2-2]
	\arrow["\lrcorner"{anchor=center, pos=0.125, rotate=180}, draw=none, from=2-2, to=1-1]
\end{tikzcd}\]
\end{thm}

\begin{proof}
In order to simplify the notation, let $G\coloneqq \pi_1(\mc{X}, x)$, $H \coloneqq \pi_1(\mc{V}, x)$ and $G_l \coloneqq \pi_1(\mc{U}_l, x)$ for $l=1,2$. The inclusions
\[\begin{tikzcd}
	{\mc{V}} & {\mc{U}_l} & {\mc{X}}
	\arrow["{j_l}", hook, from=1-1, to=1-2]
	\arrow["j", curve={height=-18pt}, hook, from=1-1, to=1-3]
	\arrow["{i_l}", hook, from=1-2, to=1-3]
\end{tikzcd}\]
give rise to a commutative diagram
\[\begin{tikzcd}
	H & {G_1} \\
	{G_2} & {G_1 *_H G_2} \\
	&& G,
	\arrow["{j_{1*}}", from=1-1, to=1-2]
	\arrow["{j_{2*}}"', from=1-1, to=2-1]
	\arrow["\iota_1", from=1-2, to=2-2]
	\arrow["{i_{1*}}", curve={height=-12pt}, from=1-2, to=3-3]
	\arrow["\iota_2", from=2-1, to=2-2]
	\arrow["{i_{2*}}"', curve={height=12pt}, from=2-1, to=3-3]
	\arrow["v", dashed, from=2-2, to=3-3]
\end{tikzcd}\]
where the dashed arrow can be filled out by the universal property of $G_1 *_H G_2$. We want to prove that $v$ is an isomorphism.

Consider the following $2$-commutative diagram
\[\begin{tikzcd}
	{\Cov_{\mc{X}}} &&& {G\mh\Set} \\
	\\
	{\Cov_{\mc{U}_1} \times_{\Cov_{\mc{V}}} \Cov_{\mc{U}_2}} &&& {G_1\mh\Set \times_{H\mh\Set} G_2\mh\Set,}
	\arrow["{{\theta = F_x^{\mc{X}}}}", from=1-1, to=1-4]
	\arrow["\alpha", from=1-1, to=3-1]
	\arrow["\beta"', from=1-4, to=3-4]
	\arrow["{{\varphi =F_x^{\mc{U}_1} \, \times_{F_x^{\mc{V}}} \, F_x^{\mc{U}_2}}}", from=3-1, to=3-4]
\end{tikzcd}\]
where $\alpha$ is induced by the restriction of covering spaces and $\beta$ is induced by $i_{1*}$ and $i_{2*}$.
Notice that $\alpha$ is an equivalence of categories, obtained by the gluing of covering spaces, $\varphi$ and $\theta$ are equivalences, as a consequence of \cite[Theorem 18.19]{Noohi-foundations}. Therefore, $\beta$ is an equivalence as well.

Notice moreover that the homomorphisms $(\iota_1, \iota_2)$ induce an equivalence of categories
\[
\gamma \colon (G_1 *_H G_2)\mh\Set \, \to \, G_1\mh\Set \, \times_{H\mh\Set} \,G_2\mh\Set,
\]
making the following diagram $2$-commute
\[\begin{tikzcd}
	{G\mh\Set} \\
	& {(G_1 *_H G_2)\mh\Set} \\
	{G_1\mh\Set \times_{H-set} G_2\mh\Set.}
	\arrow["{v^*}", from=1-1, to=2-2]
	\arrow["\beta"', from=1-1, to=3-1]
	\arrow["\gamma", from=2-2, to=3-1]
\end{tikzcd}\]
Since both $\beta$ and $\gamma$ are equivalences, this shows that $v^*$ is also an equivalence and therefore, see \cite[Proposition 4.7.1]{Douady2020}, that $v$ is an isomorphism of groups.
\end{proof}

\printbibliography
\end{document}